\documentclass[12pt,reqno]{amsart}

\usepackage{amsmath,amsfonts,amsthm,amssymb,color}
\usepackage{pdfsync}
\usepackage[latin1]{inputenc}
\usepackage{pstricks}
\usepackage{mathrsfs}
\usepackage{units}
\usepackage{graphicx}

\newtheorem{prop}{Proposition}[section]
\newtheorem{cor}[prop]{Corollary}
\newtheorem{corollary}[prop]{Corollary}

\newtheorem{lemma}[prop]{Lemma}

\newtheorem{remark}[prop]{Remark}
\newtheorem{thm}[prop]{Theorem}

\newtheorem{defi}[prop]{Definition}
\newtheorem{definition}[prop]{Definition}

\newcommand{\cont}[1]{\displaystyle{\mathop{\frown}^{#1}}}

\renewcommand{\geq}{\geqslant}
\def\leq{\leqslant}

\newcommand{\N}{\mathbb{N}}

\newcommand{\R}{\mathbb{R}}

\newcommand{\C}{\mathbb{C}}

\newcommand{\pt}[1]{\stackrel{#1}{\frown}}
\newcommand{\teti}{\mathscr{T}}

\def\e{\varepsilon}

\def\1{{\mathbf{1}}}

\def\1{{\mathbf{1}}}
\def\0.5{{\frac{1}{2}}}

\newcommand{\lla}{\left\langle}
\newcommand{\rra}{\right\rangle}

\newcommand{\lp}{\left(}
\newcommand{\rp}{\right)}
\newcommand{\lc}{\left[}
\newcommand{\rc}{\right]}

\newcommand{\ca}{\mathcal{A}}
\newcommand{\path}{{\bf r}}

\newcommand{\fin}
{ \vspace{-0.6cm}
\begin{flushright}
\mbox{$\Box$}
\end{flushright}
\noindent }

\begin{document}

%\maketitle

\begin{center}
{\large\textbf{
Convergence of Wigner integrals to the tetilla law
}}\\~\\
Aur\'elien Deya\footnote{Institut \'Elie Cartan, Universit\' e
Henri Poincar\' e, BP 70239, 54506 Vandoeuvre-l\`es-Nancy, France. Email: {\tt Aurelien.Deya@iecn.u-nancy.fr}}
and Ivan Nourdin\footnote{Institut \'Elie Cartan, Universit\' e
Henri Poincar\' e, BP 70239, 54506 Vandoeuvre-l\`es-Nancy, France. Email: {\tt inourdin@gmail.com}.
Supported in part by the two following (french) ANR grants: `Exploration des Chemins Rugueux'
[ANR-09-BLAN-0114] and `Malliavin, Stein and Stochastic Equations with Irregular Coefficients'
[ANR-10-BLAN-0121].}
\end{center}

\bigskip

{\small \noindent {\bf Abstract:} If $x$ and $y$ are two free semicircular random variables in
a non-commutative probability space $(\mathcal{A},E)$ and have variance one, we call the law of $\frac{1}{\sqrt{2}}(xy+yx)$ the {\it tetilla law}
(and we denote it by $\teti$),
because of the similarity between the form of its density and the shape of the tetilla cheese from Galicia.
In this paper,
we prove that a unit-variance sequence $\{F_n\}$ of multiple
Wigner integrals converges in distribution to $\mathscr{T}$
if and only if $E[F_n^4]\to E[\teti^4]$ and $E[F_n^6]\to E[\teti^6]$.
This result should be compared with limit theorems of the same flavor, recently obtained by Kemp, Nourdin, Peccati \& Speicher
\cite{knps} and Nourdin \& Peccati \cite{np-poisson}.
\bigskip


\bigskip

 }

\section{Introduction}

In a seminal paper of 2005, Nualart and Peccati discovered the following fact, called the {\it Fourth Moment Theorem} in the sequel:
for a sequence of normalized multiple
Wiener-It\^o integrals to converge to the standard Gaussian law, it is (necessary and) sufficient that its fourth moment tends
to 3. This somewhat suprising result has been the starting point of a new line of research, and has quickly led to several
applications, extensions and improvements in various areas and directions, including: Berry-Esseen type's inequalities \cite{np-ptrf}
with sometimes optimal bounds
\cite{np-aop,bbnp}, further developments in the multivariate case \cite{nrp,amv},
second order Poincaré inequalities \cite{npr-jfa}, new expression for
the density of functionals of Gaussian field \cite{nv-ejp}, or universality results for homogeneous sums \cite{npr-aop},
to cite but a few.
We refer the reader to the forthcoming monograph \cite{NPbook} for an overview of the most important developments;
see also \cite{WWW} for a constantly  updated web resource, with links to all available papers on the subject.

In this paper, we are more specifically concerned with the possible extensions of the Fourth Moment Theorem in the context of {\it free probability}.
This theory, popularized by Voiculescu \cite{voiculescu} in the early 1990's, admits the so-called {\it free Brownian motion}
as a central object. Free Brownian motion may be seen as
an infinite-dimensional symmetric matrix-valued Brownian motion and,
exactly as classical Brownian motion allows to express limits arising from random walks (Donsker's theorem),
the former enables to describe many limits involving traces of random matrices whose size tends to infinity.
It is actually defined in a very similar fashion to its classical counterpart, the only notable difference being that
its increments are {\it freely} independent and are distributed according to the Wigner's {\it semicircular law}
(see Definition \ref{defi:freeBM} for the details).

By mimicking the classical construction of {\it multiple Wiener-It\^o integrals} (see e.g. the book \cite{Nualart-book}
by Nualart, which is the classical reference on this subject), one can
now define the so-called {\it Wigner multiple integral}, as was done by Biane and Speicher in \cite{bianespeicher} and whose
construction is
recalled in Section \ref{s:wigner}.
(The terminology `Wigner integral' was invented in a humorous nod to the fact that Wigner's semicircular law plays the central role
here, and the similarity between the names Wigner and Wiener.)
As such,
this gives a precise meaning to the following object, called the $q$th {\it Wigner multiple integral with kernel} $f$:
\begin{equation}\label{iqf}
I_q^{(S)}(f)=\int_{\R_+^q} f(t_1,\ldots,t_q)dS_{t_1}\ldots dS_{t_q},
\end{equation}
when $q\geq 1$ is an integer, $f\in L^2(\R_+^q)$ is a deterministic function,
and $S=(S_t)_{t\geq 0}$ is a free Brownian motion.

If one considers a classical Brownian motion $B=(B_t)_{t\geq 0}$ instead of $S$ in (\ref{iqf}),
one gets the Wiener-It\^o multiple integral $I_q^{(B)}(f)$
with kernel $f$. In this case, it is well-known that we can restrict ourselves to {\it symmetric} kernels $f$
(that is, satisfying $f(t_1,\ldots,t_q)=f(t_{\sigma(1)},\ldots,t_{\sigma(q)})$ for almost all $t_1,\ldots,t_q\in\R_+$ and
all permutation $\sigma\in\mathfrak{S}_q$) without loss of generality, and that
the following
multiplication formula is in order: if $f\in L^2(\R_+^p)$ and $g\in L^2(\R_+^q)$ are both symmetric, then
\[
I_p^{(B)}(f)I_q^{(B)}(g)=\sum_{r=0}^{p\wedge q}
r!\binom{p}{r}\binom{q}{r} I_{p+q-2r}^{(B)}(f\cont{r} g),
\]
where $f\cont{r} g\in L^2(\R_+^{p+q-2r})$ is the $r$th {\it contraction} of $f$ and $g$, given by
\begin{eqnarray}\label{contr}
&&f\cont{r} g (t_1,\ldots,t_{p+q-2r}) \\
&=& \int_{\R_+^{p+q-2r}} f(t_1,\ldots,t_{p-r},x_1,\ldots,x_r)g(x_r,\ldots,x_1,t_{p-r+1},\ldots,t_{p+q-2r})
dx_1\ldots dx_r,\notag
\end{eqnarray}
$t_1,\ldots,t_{p+q-2r}\in\R_+.$

In the definition (\ref{contr}), there is of course no incidence to permute the variables inside
$f$ or inside $g$, thanks to the symmetry assumption.
In contrast, one must warn the reader
that the same may have dramatic consequences in the free context: indeed, because the increments of $S$
do not commute, permuting the variables inside $f$ generally changes the value of
$I_p^{(S)}(f)$. A bit surprisingly however, it turns out that the multiplication of two multiple Wigner integrals
takes a simpler form compared to the classical case.
Precisely, if $f\in L^2(\R_+^p)$ and $g\in L^2(\R_+^q)$, then
\[
I_p^{(S)}(f)I_q^{(S)}(g)=\sum_{r=0}^{p\wedge q} I_{p+q-2r}^{(S)}(f\cont{r} g).
\]
To understand more deeply the similarities and differences between the multiplication formulae in the free and in the classical settings,
we refer the reader to the paper \cite{donati} by Donati-Martin, where such a product formula is more generally derived
for the $q$-Brownian motion, which is nothing but an interpolation between the classical Brownian motion ($q=1$) and the free
Brownian motion ($q=0$).

Let us now go back to the heart of this paper.
Very recently, Kemp, Nourdin, Peccati and Speicher \cite{knps} extended the Fourth Moment Theorem
to the free setting:
this time, for a sequence of normalized Wigner multiple integrals to converge to the semicircular law, it is
(necessary and) sufficient that the fourth moment tends to 2, which is of course the value
of the fourth moment of the semicircular law.
Shortly afterwards,
Nourdin and Peccati considered in \cite{np-poisson} the problem of determining,
still in term of a finite number of moments, the convergence of a sequence of multiple Wigner integrals to
the {\it free Poisson} distribution (also called the Marchenko-Pastur distribution).
In this case, what happens to be necessary and sufficient is not only the convergence of the fourth moment, but also the convergence
of the third moment as well.
(Actually, only the convergence of a linear combination of these two moments turns out to be needed.)

In the present paper, our goal is to go one step further with respect to the two previously quoted papers \cite{knps,np-poisson},
by studying yet another distribution for which a similar
phenomenom occurs. More precisely, as a target limit we consider the random variable $\frac{1}{\sqrt{2}}(xy+yx)$, where
$x$ and $y$ are two free, centered, semicircular random variables with unit variance. We decided to call its law the {\it tetilla law},
because of the troubling similarity between the shape of its density and the form of the tetilla cheese from
Galicia, see the forthcoming Section \ref{s:tetilla}
for further details together with some pictures.
After our paper was submitted, it was brought to our notice that the tetilla law already appeared in the
reference \cite{duke} (see Example 1.5(2) therein) by Nica and Speicher, under the more conventional name
``symmetric Poisson''.

Here is now the precise question we aim to solve in the present paper:
is it possible, by means of a finite number of their moments only, to characterize the convergence
to the tetilla law of a given unit-variance sequence of multiple Wigner integrals? If so,
how many moments are then needed? (and what are they?)

As we will see, the answer to our first question is positive; furthermore,
it turns out that, unlike
the known related papers \cite{knps,np-poisson} in the literature, the novelty is that we must here have the convergence
of both
the fourth and the {\it sixth} moments to get the desired conclusion.
More specifically, we shall prove the following result in the present paper.
\begin{thm}\label{mainthm}
Let $\teti$ be a random variable distributed according to the tetilla law, and fix an integer $q\geq 2$.
Let $F_n=I_q(f_n)$ be a sequence of Wigner multiple integrals, where each $f_n$ is an 
element 
of $L^2(\R_+^q)$ such that $\|f_n\|_{L^2(\R_+^q)}=1$ and $f_n(t_1,\ldots,t_q)=f_n(t_q,\ldots,t_1)$ for
almost all $t_1,\ldots,t_q\in\R_+$. Then, the following three assertions are equivalent
as $n\to\infty$:
\begin{enumerate}
\item[(i)] $E[F_n^6]\to E[\teti^6]$ and $E[F_n^4]\to E[\teti^4]$;
\item[(ii)] For all $r,r'=1,\ldots,q-1$ such that $r'+2r\leq 2q$ and $r+r'\neq q$, it holds true that
\begin{equation}\label{contract-1}
(f_n\cont{r} f_n)\cont{r'}f_n\to 0
\end{equation}
and
\begin{equation}\label{contract-2}
-\frac12 f_n + \sum_{r=1}^{q-1} (f_n\cont{r}f_n)\cont{q-r}f_n\to 0;
\end{equation}
\item[(iii)] $F_n\to \teti$ in distribution.
\end{enumerate}
\end{thm}

At first glance, one could legitimately think that, in order to show our Theorem \ref{mainthm}, the only
thing to do is somehow to merely
extend the existing techniques introduced in \cite{knps,yetanother,np-poisson}. This is actually not the case.
Although the overall philosophy of the proof remains similar (more precisely, we shall follow the strategy introduced in \cite{yetanother}), here we
need to rely on new identities about iterated contractions
(such as the string of contractions appearing in (ii)) to be able to conclude, and we consider that the discovery of these
crucial identities represents one of the main achievement
of our study.

To finish this introduction, we offer the following result as a corollary of Theorem \ref{mainthm}. We believe that
it has its own interest
because, for the time being, very few is known about the laws which are admissible for a multiple Wigner integrals
with a given order.

\begin{cor}\label{notetilla}
Let $q\geq 3$ be an integer, and let $f$ be an
element of $L^2(\R_+^q)$ 
such that $\|f\|_{L^2(\R_+^q)}=1$
and $f(t_1,\ldots,t_q)=f(t_q,\ldots,t_1)$ for
almost all $t_1,\ldots,t_q\in\R_+$.
Then, $I_q(f)$ cannot be distributed according to the tetilla law.
\end{cor}

The rest of the paper is organised as follows. Section 2 contains some useful preliminaries and, among other things,
introduce the reader to the tetilla law. Then, the proof of Theorem \ref{mainthm} is done in Section 3.
Finally, Section 4 corresponds to an appendix, where the proofs of two technical results have been postponed.

\section{Preliminaries}

\subsection{The free probability setting} Our main reference for this section
is the monograph by Nica and Speicher \cite{nicaspeicher},
to which the reader is referred for any unexplained notion or result.

For the rest of the paper, we consider as given a so-called (tracial) {\it  $W^*$-probability space} $(\mathscr{A},E)$, where:
$\mathscr{A}$ is a von Neumann algebra of operators (with involution $x\mapsto x^*$),
and $E$ is a unital linear functional on $\mathscr{A}$ with the properties of being {\it weakly continuous}, {\it positive}
(that is, $E(xx^*)\geq 0$ for every $x\in \mathscr{A}$), {\it faithful} (that is, such that the relation $E(xx^*) = 0$ implies $x=0$),
and {\it tracial} (that is, $E(xy) = E(yx)$, for every $x,y\in \mathscr{A}$).

As usual in free probability, we refer to the self-adjoint elements of $\mathscr{A}$ as {\it random variables}. Given a random variable $x$
we write $\mu_x$ to indicate the {\it law} (or {\it distribution}) of $x$, which is defined as the unique Borel probability measure on
$\R$ such that $E(x^m) = \int_\R t^m d\mu_x(t)$ for every integer $m\geq 0$ (see e.g. \cite[Proposition 3.13]{nicaspeicher}).

We say that the unital subalgebras $\mathscr{A}_1,...,\mathscr{A}_n$ of $\mathscr{A}$ are {\it freely independent} whenever the
following property holds: let $x_1,...,x_m$ be a finite collection of elements chosen among the $\mathscr{A}_i$'s in such a way that
(for $j=1,...,m-1$) $x_j$ and $x_{j+1}$ do not come from the same $\mathscr{A}_i$ and $E(x_j) = 0$ for $j=1,...,m$; then
$E(x_1\ldots x_m)=0$. Random variables are said to be freely independent if they generate freely independent unital subalgebras of
$\mathscr{A}$.

\subsection{Free cumulants, $R$-transform and Cauchy transform}\label{s:cumulants}
Given an integer $m \geq 1$, we write $[m] = \{1,...,m\}$.
A {\it partition} of $[m]$ is a collection of non-empty and disjoint subsets of $[m]$, called {\it blocks},
such that their union is equal to $[m]$.

A partition $\pi$ of $[m]$ is said to be {\it non-crossing} if one cannot find integers $p_1,q_1,p_2,q_2$ such that:
(a) $1\leq p_1 < q_1 < p_2 < q_2\leq m$,
(b) $p_1,p_2$ are in the same block of $\pi$,
(c)  $q_1,q_2$ are in the same block of $\pi$, and
(d) the $p_i$'s are not in the same block of $\pi$ as the $q_i$'s.
The collection of the non-crossing partitions of $[m]$ is denoted by $NC(m)$, $m\geq 1$.

Given a random variable $x$, we denote by $\{\kappa_m(x) : m\geq 1\}$ the sequence of its {\it free cumulants}:
according to \cite[p. 175]{nicaspeicher}, they are defined through the recursive relation
\begin{equation}\label{e:momcum}
E(x^m) = \sum_{\pi = \{b_1,...,b_j\}\in NC(m)} \,\,\prod_{i=1}^j\kappa_{|b_i|}(x),
\end{equation}
where $|b_i|$ indicates the cardinality of the block $b_i$ of the non-crossing partition $\pi$.
The sequence  $\{\kappa_m(x) : m\geq 1\}$ completely determines the moments of
$x$ (and vice-versa), and the power series
\[
R_x(z)=\sum_{m=1}^\infty \kappa_m(x)z^m,
\]
is called the $R$-transform of (the distribution of) $x$. Its main properties is that it linearizes free convolution, just
as the classical cumulant transform linearizes classical convolution: that is, if $x$ and $y$ are free random variables,
then $R_{x+y}=R_x+R_y$ (as a formal series).

To recover the distribution $\mu_x$ from the free cumulants of the random variable $x$,
it is common to use its Cauchy transform $G_x$. It is defined by
\[
G_x(z)=\int_\R  \frac{d\mu_x(t)}{z-t},\quad z\in\mathbb{C}_+=\{z\in\mathbb{C}:\,{\rm Im}z>0\},
\]
and takes its values in $\mathbb{C}_-=\{z\in\mathbb{C}:\,{\rm Im}z<0\}$.
The Cauchy transform can be found from the $R$-transform as the inverse function of
$z\mapsto \frac1z\big(1+R_x(z)\big)$, that is, it verifies
\[
G_x\left[\frac1z\big(1+R_x(z)\big)\right]=z.
\]
On the other hand, Stieltjes inversion theorem states that
\begin{equation}\label{dmu}
d\mu_x(t) = -\frac1\pi\,\lim_{\e\to 0} \,\,{\rm Im}\big[G_x(t+i\e)\big]dt,\quad t\in\R,
\end{equation}
where the limit is to be understood in the weak topology on the space of probability measures on $\R$.

\subsection{Semicircular law}\label{S:semicircular}

The following family of distributions is fundamental in free probability. It plays the same role as the Gaussian
laws
in a classical probability space.

\begin{defi}\label{D:Z}
The centered {\it semicircular distribution} of parameter $t>0$, denoted by $S(0,t)$, is the probability distribution
which is absolutely continuous with respect to the Lebesgue measure, and whose density is given by
\[
p_t(u) = \frac{1}{2\pi t}\,\sqrt{4t-u^2}\,{\bf 1}_{(-2\sqrt{t},2\sqrt{t})}(u).
\]
\end{defi}
One can readily check that
$
\int_{-2\sqrt{t}}^{2\sqrt{t}} u^{2m} p_t(u)du = C_m t^m,
$
where $C_m$ is the $m$th Catalan number (so that e.g. the second moment of $S(0,t)$ is $t$).
One can deduce from the previous relation and
(\ref{e:momcum}) (e.g. by recursion) that the free cumulants of a random variable $x$ with law $S(0,t)$ are all
zero, except for $\kappa_2(x) =E[x^2]= t$ (equivalently, the $R$-transform of $x$ is given by $R_x(z)=tz^2$).
Note also that $S(0,t)$ is compactly supported, and therefore is uniquely determined by their moments
(by the Weierstrass theorem).

On the other hand, it is a classical fact (see e.g. \cite[Proposition 12.13]{nicaspeicher})
that the free cumulants of $x^2$, whenever $x\sim\ S(0,t)$, are given by
\begin{equation}\label{freepoisson}
\kappa_m(x^2)=t^m, \quad m\geq 1.
\end{equation}

\subsection{Free Brownian motion and Wigner chaoses}\label{s:wigner}
Our main reference for the content of this section is the paper by Biane and Speicher \cite{bianespeicher}.

\begin{defi}\label{d:spaces}{\rm
\begin{itemize}
\item[\rm (i)] For $1\leq p \leq \infty $, we write $L^p(\mathscr{A},E)$ to indicate the $L^p$ space obtained as the completion of
$\mathscr{A}$ with respect to the norm $\| a\|_p = E(|a|^p)^{1/p}$, where $ |a|= \sqrt{a^\ast a}$, and $\|\cdot\|_\infty$ stands for the operator norm.
\item[\rm (ii)] For every integer $q\geq 2$, the space $L^2(\mathbb{R}_+^q)$ is the collection of all real-valued
functions on $\mathbb{R}_+^q$ that are square-integrable with respect to the Lebesgue measure.
We use the short-hand notation $\langle\cdot,\cdot\rangle_q$ to indicate the inner product in $L^2(\mathbb{R}_+^q)$.
Also, given $f\in L^2(\mathbb{R}_+^q)$, we write
$f^*(t_1,t_2,...,t_q) = f(t_q,...,t_2,t_1)$,
and we call $f^*$ the {\it adjoint} of $f$. We say that an element of $L^2(\mathbb{R}_+^q)$ is {\it mirror symmetric} whenever
$f = f^*$ as a function.
\item[\rm (iii)] Given $f\in L^2(\mathbb{R}_+^q)$ and $g\in L^2(\mathbb{R}_+^p)$ for every $r = 1,...,\min(q,p)$,;
we define the $r$th {\it contraction} of $f$ and $g$ as the element of $L^2(\mathbb{R}_+^{p+q-2r})$ given by
(\ref{contr}).
One also writes $f\cont{0} g (t_1,...,t_{p+q}) = f\otimes g (t_1,...,t_{p+q}) = f(t_1,...,t_q)g(t_{q+1},...,t_{p+q})$. In the following, we shall use the notations  $f\cont{0} g$ and $f\otimes g$ interchangeably. Observe that, if $p=q$, then $f\cont{p} g = \langle f,g^{*}\rangle_{L^2(\R_+^q)}$.
\end{itemize}
}
\end{defi}

Let us now define what a free Brownian motion is.

\begin{defi}\label{defi:freeBM}
A {\it free Brownian motion} $S$ on $(\mathscr{A},E)$ consists of: (i) a filtration $\{\mathscr{A}_t : t\geq 0\}$ of von Neumann sub-algebras of $\mathscr{A}$ (in particular, $\mathscr{A}_u \subset \mathscr{A}_t$, for $0\leq u<t$),
(ii) a collection $S = (S_t)_{t\geq 0}$ of self-adjoint operators such that:
\begin{itemize}
\item[--] $S_t\in\mathscr{A}_t$  for every $t$;
\item[--] for every $t$, $S_t$ has a semicircular distribution $S(0,t)$;
\item[--] for every $0\leq u<t$, the increment $S_t - S_u$ is freely independent of $\mathscr{A}_u$, and has a semicircular distribution
$S(0,t-u)$.
\end{itemize}
\end{defi}

For every integer $q\geq 1$, the collection of all random variables of the type $I^{(S)}_q(f) = I_q(f)$,
$f \in L^2(\mathbb{R}_+^q)$, is called the $q$th {\it Wigner chaos} associated with $S$, and is defined according to
\cite[Section 5.3]{bianespeicher}, namely:
\begin{itemize}
\item[--] first define $I_q(f) = (S_{b_1} - S_{a_1})\ldots (S_{b_q} - S_{a_q})$, for every function $f$ having the form
\begin{equation}\label{e:simple}
f(t_1,...,t_q) = {\bf 1}_{(a_1,b_1)}(t_1)\times\ldots\times {\bf 1}_{(a_q,b_q)}(t_q),
\end{equation}
where the intervals $(a_i,b_i)$, $i=1,...,q$, are pairwise disjoint;
\item[--] extend linearly the definition of $I_q(f)$ to simple functions vanishing on diagonals, that is, to functions $f$ that are finite
linear combinations of indicators of the type (\ref{e:simple});
\item[--] exploit the isometric relation
\begin{equation}\label{e:freeisometry}
\langle I_q(f_1),I_q(f_2) \rangle_{L^2(\mathscr{A},E )}=
E\left[I_q(f_1)^*I_q(f_2)\right]=
E\left[I_q(f_1^*)I_q(f_2)\right]=
 \langle f_1,f_2 \rangle_{L^2(\mathbb{R}_+^q)},
\end{equation}
where $f_1,f_2$ are simple functions vanishing on diagonals, and use a density argument to define $I_q(f)$ for a general
$f\in  L^2(\mathbb{R}_+^q)$.
\end{itemize}

Observe that relation (\ref{e:freeisometry}) continues to hold for every pair $f_1,f_2 \in L^2(\mathbb{R}_+^q)$. Moreover, the above sketched
construction implies that $I_q(f)$ is self-adjoint if and only if $f$ is mirror symmetric. We recall the following fundamental
multiplication formula, proved in \cite{bianespeicher}. For every $f\in L^2(\R_+^p)$ and $g\in L^2(\R_+^q)$, where $p,q\geq 1$,
\begin{equation}\label{e:mult}
I_p(f)I_q(g) = \sum_{r=0}^{p\wedge q} I_{p+q-2r}(f\cont{r}g).
\end{equation}

By applying (\ref{e:mult}) iteratively and by taking into account that a $n$th Wigner integral ($n\geq 1$) is centered
by construction,
we immediately get the following formula, that relates explicitely the moments of a multiple
Wigner integral to its kernel.
\begin{corollary}\label{cor:for-mo}
For every function $f\in L^2(\R_+^q)$ and every integer $l\geq 2$, one has
\begin{equation}\label{form-gene-moments}
E\lc I_q(f)^l \rc =\sum_{(r_1,\ldots,r_{l-1}) \in \ca_{q,l}} ((\ldots(f\pt{r_1} f) \pt{r_2} f) \pt{r_3}) \ldots \pt{r_{l-1}} f,
\end{equation}
where $\ca_{q,l}$ stands for the set of elements $(r_1,\ldots,r_{l-1}) \in \{0,\ldots,q\}^{l-1}$ which satisfies the two conditions:
\[
2r_1+\ldots+2r_{k-1}+r_k \leq kp \ \ \ \text{for every} \ k\in \{2,\ldots,l-1\}, \quad \text{and} \quad
2r_1+\ldots+2r_{l-1}=lp.
\]
\end{corollary}

\subsection{Tetilla law}\label{s:tetilla}

We are now in a position to define the so-called tetilla law, which lies at the very heart of this paper.
See also \cite[Example 1.5(2)]{duke} for other properties.

\begin{definition}
Let $x,y$ be two free semicircular random variables with variance one. The law of the random variable
$\frac{1}{\sqrt{2}} \lp xy+yx \rp$ is denoted $\teti$, and is called the tetilla law.
\end{definition}

\begin{lemma}\label{lem:cum}
Let $x,y$ be two free semicircular random variables with variance one. Then the random variable
$w=\frac{1}{\sqrt{2}} \lp x^2-y^2 \rp$ is distributed according to the tetilla law. As a consequence, the free cumulants of the
tetilla law are given by \[
\kappa_m(w) =\left\{\begin{array}{cl}
2^{1-m/2}&\quad\mbox{if $m$ is even}\\
0&\quad\mbox{if $m$ is odd}
\end{array}
\right.
.
\]
Equivalently, the $R$-transform of $w$ is given by $R_w(z)=2z^2/(2-z^2)$.
\end{lemma}
\begin{proof}
It is immediately checked that the two random vectors $(x,y)$ and
$\left(\frac{x+y}{\sqrt{2}},\frac{x-y}{\sqrt{2}}\right)$ share the same law.
(Indeed, they are both jointly semicircular with the same covariance matrix, see \cite[Corollary 8.20]{nicaspeicher}.)
Consequently, the two random variables
$xy+yx$ and $x^2-y^2$
have the same law as well, thus showing that $w$ is distributed according to the tetilla law.
Thanks to this new representation and the linearization property of the $R$-transform with respect to free
convolution, it is now easy to calculate the free cumulants of $\teti$. For any $m\geq 1$, we have
\begin{eqnarray*}
\kappa_m(w)&=&2^{-m/2}\kappa_m(x^2-y^2)=2^{-m/2}\big(\kappa_m(x^2)+(-1)^m\kappa_m(y^2)\big)\\
&=&2^{-m/2}(1+(-1)^m)\kappa_m(x^2),
\end{eqnarray*}
and the desired conclusion now follows from (\ref{freepoisson}).
\end{proof}

\begin{prop}\label{prop:densite}
The tetilla law $\teti$ admits a compactly supported density $h$ with respect to the Lebesgue measure, given by
\begin{equation}\label{densite}
h(t)=
\frac{1}{2\sqrt{3}\pi t} \left[
\sqrt[3]{1+36t^2+3 \sqrt{6t^2+132t^4-24t^6} \,}
-
\sqrt[3]{1+36t^2-3 \sqrt{6t^2+132t^4-24t^6} \,}
\right],
\end{equation}
for $t\in \left[ -\frac{\sqrt{11+5\sqrt{5}}}{2},\frac{\sqrt{11+5\sqrt{5}}}{2}\right]$, $h(t)=0$ otherwise.
\end{prop}

It is formula (\ref{densite}) that motivated us to call the law of $\teti$ the tetilla law.
The reader should indeed have a look at the two following pictures, where the similarity between the graph of $h$ and the
shape of the tetilla cheese
seemed evident to us.

\bigskip

\includegraphics[width=15cm]{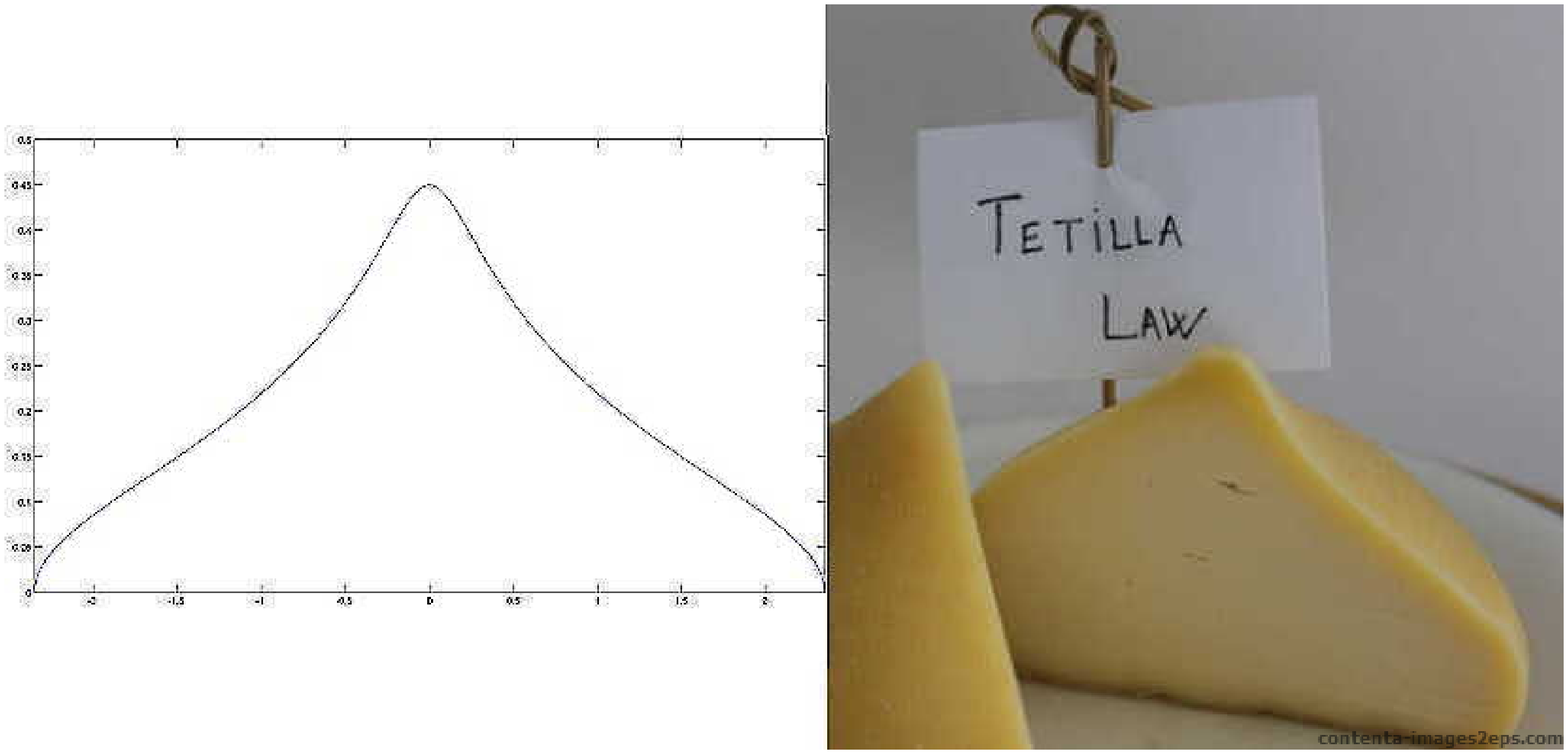}

\bigskip

\noindent
{\it Proof of Proposition \ref{prop:densite}}.
Let $w$ have the tetilla law.
According to Section \ref{s:cumulants} and using the expression of the $R$-transform given in Lemma \ref{lem:cum}, the
Cauchy transform $y=G_w(z)$ ($z\in\C_+$) of $w$ at point $z$ verifies
$z y^3+y^2-2z y+2=0$. The explicit solutions $y_1,y_2,y_3$ of the latter equation can be obtained thanks to Cardan's formulae:
for every $z\in \C_+$, one has
$$y_1(z)=-(u(z)+v(z))-\frac{1}{3z}, \,\, y_2(z)=-(j u(z)+\overline{j} v(z))-\frac{1}{3z}, \,\, y_3(z)=-(\overline{j} u(z)+j v(z))-\frac{1}{3z},$$
where $j=e^{2i\pi /3}$, and where we successively set
$$u(z)=\lp \frac{q(z)+\sqrt{\Delta(z)}}{2} \rp^{1/3}, \quad v(z)=\lp \frac{q(z)-\sqrt{\Delta(z)}}{2} \rp^{1/3},$$
$$q(z)=\frac{2}{27 z^3} (1+18 z^2), \quad \Delta(z)=q(z)^2+\frac{4}{27} p(z)^3, \quad p(z)=-(1+\frac{1}{3z^2}).$$
Now, in order to identify $G_w$ with one of these solutions, let us observe that both $\text{Im} \, y_1(i)$ and $\text{Im} \, y_3(i)$ are strictly positive reals (this is only straightforward computation).
Since $G_w$ is known to take its values in $\C_-$, we can assert that $G_w=y_2$ on $\C_+$. The density (\ref{densite}) is
then easily derived from the Stieltjes inversion formula (\ref{dmu}).

\fin

\subsection{Moments}

Once endowed with its free cumulants, we can go back to Formula (\ref{e:momcum}) in order to compute the moments of the tetilla law.

\begin{prop}
The moments of the tetilla law are given by the following formulae: for every $n\geq 1$,
\begin{equation}\label{mom-exp}
m_{2n+1}(\teti)=0 \quad  , \quad m_{2n}(\teti)=\frac{1}{2^n n} \sum_{k=1}^n 2^k {2n \choose{k-1}}{n\choose{k}}.
\end{equation}
\end{prop}

\begin{proof}
The fact that the odd moments are all equal to zero is a direct consequence of the symmetry of the density $h$, see (\ref{densite}). Recall next that $\kappa_{2k+1}(\teti)=0$ and $\kappa_{2k}(\teti)=2^{-k+1}$. As a consequence,
\begin{eqnarray*}
m_{2n}(\teti)&=& \sum_{k=1}^{2n} \sum_{(b_1,\ldots,b_k)\in NC(2n)} \prod_{i=1}^k \kappa_{|b_i|}(\teti)
=\sum_{k=1}^{n} \sum_{\substack{(b_1,\ldots,b_k)\in NC(2n)\\ |b_1|,\ldots,|b_k| \, \text{all even}}} \prod_{i=1}^k 2^{-\frac{|b_i|}{2}+1}\\
&=&\sum_{k=1}^{n} \sum_{\substack{(b_1,\ldots,b_k)\in NC(2n)\\ |b_1|,\ldots,|b_k| \, \text{all even}}} 2^k 2^{-\frac{|b_1|+\ldots+|b_k|}{2}}\\
&=&\frac{1}{2^n} \sum_{k=1}^n 2^k \, \text{Card}\{(b_1,\ldots,b_k) \in NC(2n): \, |b_1|,\ldots,|b_k| \, \text{all even}\}.
\end{eqnarray*}
It turns out that the latter cardinal has been explicity computed in \cite[Lemma 4.1]{edelman}:
$$\text{Card}\{(b_1,\ldots,b_k) \in NC(2n): \, |b_1|,\ldots,|b_k| \, \text{all even}\}={2n \choose{k-1}}{n\choose{k}},$$
which gives the result.
\end{proof}

\subsection{An induction formula for the moments of the tetilla law}\label{subsec:moments-tetilla}
Now, with the help of Formula (\ref{form-gene-moments}), we are going to exhibit a specific algorithm that governs the sequence of the moments $(m_k(\teti))$. The purpose here does not lie in getting a way to compute these moments explicitly (one can use Formula (\ref{mom-exp}) to do so). In fact, the algorithm will be used as a guideline through the proof of our Theorem \ref{mainthm}.

%This section is devoted to the computation of the moments of the tetilla law.
%Because this is not adapted to what
%we will need for our proof of Theorem \ref{mainthm},
%we will not use the explicit form of the density; instead, we will set up an induction procedure.
%Before being in a position to do so, we introduce some additional notation.

\medskip

\noindent
\textbf{Notation}. For integers $p,q,l \geq 2$, let $\ca_{q,p,l}$ be the set of elements
$\path_{l-1}=\{r_1,\ldots,r_{l-1}\} \in \{0,\ldots,q\}^{l-1}$ which satisfy the following two conditions:

\smallskip

$(a)$ $2r_1+\ldots+2r_{k-1}+r_k \leq (k-1)q+p$ for any $k\in \{1,\ldots,l-1\}$;

\smallskip

$(b)$ $2r_1+\ldots+2r_{l-1}=(l-1)q+p$.

\noindent
Otherwise stated, $\ca_{q,p,l}$ stands for the set of elements
$\path_{l-1}=\{r_1,\ldots,r_{l-1}\} \in \{0,\ldots,q\}^{l-1}$ for which the $(l-1)$-th iterated contraction
starting from some $g\in L^2(\R_+^p)$ and continuing with some $f\in L^2(\R_+^q)$, that is,
\begin{equation}\label{nota-contr}
\mathcal{C}_{f,\path_{l-1}}(g):=(\ldots((g \pt{r_1} f) \pt{r_2} f) \pt{r_3} \ldots) \pt{r_{l-1}} f,
\end{equation}
is well-defined (condition $(a)$) and reduces to a real number (condition $(b)$). In particular, with the notation of Corollary \ref{cor:for-mo}, one has $\ca_{q,l}=\ca_{q,q,l}$.

\smallskip

With every $\path_{l-1} \in \ca_{q,p,l}$, one associates a walk $M$ on $\{0,\ldots,l-1\}$ as follows:
$$M_0=p, \quad M_k=kq+p-2r_1-\ldots-2r_{k}, \quad k=1,\ldots,l-1.$$
When dealing with functions $f\in L^2(\R_+^q)$ and $g\in L^2(\R_+^p)$ as before, $M_k$ corresponds to the number of arguments of
the function $(\ldots((g\pt{r_1} f) \pt{r_2} f) \pt{r_3}) \ldots \pt{r_k} f$. Then, the above conditions $(a)$ and $(b)$ can be
translated into the following constraints on the walk $M$:

\smallskip

$(i)$ $M_0=p$, $M_{l-1}=0$ and $M_k \geq 0$ for all $k=0,\ldots,l-1$;

\smallskip

$(ii)$ $M_{k+1}-M_k \in \{-q,-q+2,\ldots, q-2,q\}$;

\smallskip

$(iii)$ if $M_k \leq q$, then $M_{k+1} \geq q-M_k$.

\smallskip

Inversely, it is easily seen that, when a walk $M$ on $\{0,\ldots,l-1\}$ satisfying $(i)-(iii)$ is given,
one can recover a (unique) element $\path_{l-1} \in \ca_{q,p,l}$ by setting
\begin{equation}\label{bijection}
r_k=\frac{1}{2} \lp q-M_k+M_{k-1}\rp , \quad k=1,\ldots,l-1.
\end{equation}
This bijection gives us a handy graphical representation for the elements of $\ca_{q,p,l}$ (see Figure 1).
Also, by a slight abuse of notation, $\ca_{q,p,l}$ will also refer in the sequel to the set of walks on $\{0,\ldots , l-1\}$
subject to the constraints $(i)-(iii)$.

\begin{center}
\begin{figure}[!ht]\label{pic:walk}

\begin{pspicture}(0,0)(15,10)

%\psgrid

\psline(0,-0.5)(0,10)

\psline(-0.5,0)(15,0)

\psline(-0.1,1)(0.1,1) \psline(-0.1,2)(0.1,2) \psline(-0.1,3)(0.1,3) \psline(-0.1,4)(0.1,4) \psline(-0.1,5)(0.1,5) \psline(-0.1,6)(0.1,6) \psline(-0.1,7)(0.1,7) \psline(-0.1,8)(0.1,8) \psline(-0.1,9)(0.1,9) \psline(-0.1,10)(0.1,10)

\rput(-0.3,1){\tiny{$2$}} \rput(-0.3,2){\tiny{$4$}} \rput(-0.3,3){\tiny{$6$}} \rput(-0.3,4){\tiny{$8$}} \rput(-0.3,5){\tiny{$10$}}

\psline(0,-0.1)(0,0.1) \psline(1.5,-0.1)(1.5,0.1) \psline(3,-0.1)(3,0.1) \psline(4.5,-0.1)(4.5,0.1) \psline(6,-0.1)(6,0.1) \psline(7.5,-0.1)(7.5,0.1) \psline(9,-0.1)(9,0.1) \psline(10.5,-0.1)(10.5,0.1) \psline(12,-0.1)(12,0.1) \psline(13.5,-0.1)(13.5,0.1)

\psline(0,2)(1.5,1) \psline(1.5,1)(3,1) \psline(3,1)(4.5,3) \psline(4.5,3)(6,4) \psline(6,4)(7.5,2) \psline(7.5,2)(9,0) \psline(9,0)(10.5,2) \psline(10.5,2)(12,2) \psline(12,2)(13.5,0)

\psline[linestyle=dotted](12,2)(6,10) \psline[linestyle=dotted](6,10)(0,2) \psline[linestyle=dotted](0,2)(1.5,0)

\rput(-0.3,-0.3){\tiny{$0$}} \rput(1.5,-0.3){\tiny{$1$}} \rput(3,-0.3){\tiny{$2$}} \rput(4.5,-0.3){\tiny{$3$}} \rput(6,-0.3){\tiny{$4$}} \rput(7.5,-0.3){\tiny{$5$}} \rput(9,-0.3){\tiny{$6$}} \rput(10.5,-0.3){\tiny{$7$}} \rput(12,-0.3){\tiny{$8$}} \rput(13.5,-0.3){\tiny{$9$}}

\rput(0.75,1.75){\tiny{$3$}} \rput(2.25,1.2){\tiny{$2$}} \rput(3.75,2.3){\tiny{$0$}} \rput(5.25,3.8){\tiny{$1$}} \rput(6.75,3.3){\tiny{$4$}} \rput(8.25,1.3){\tiny{$4$}} \rput(9.75,1.3){\tiny{$0$}} \rput(11.25,2.2){\tiny{$2$}} \rput(12.75,1.3){\tiny{$4$}}

\end{pspicture}

\

\caption{A walk in $\ca_{4,4,10}$. The number associated to each $[M_k,M_{k+1}]$ corresponds to the contraction order $r_{k+1}$. Observe in particular that every walk $M$ in $\ca_{4,4,10}$ is contained in the area delimited by the dotted line. Besides, $M_8$ is forced to be $4$. More generally, $M_{l-2}$ is necessarily equal to $q$.}
\end{figure}
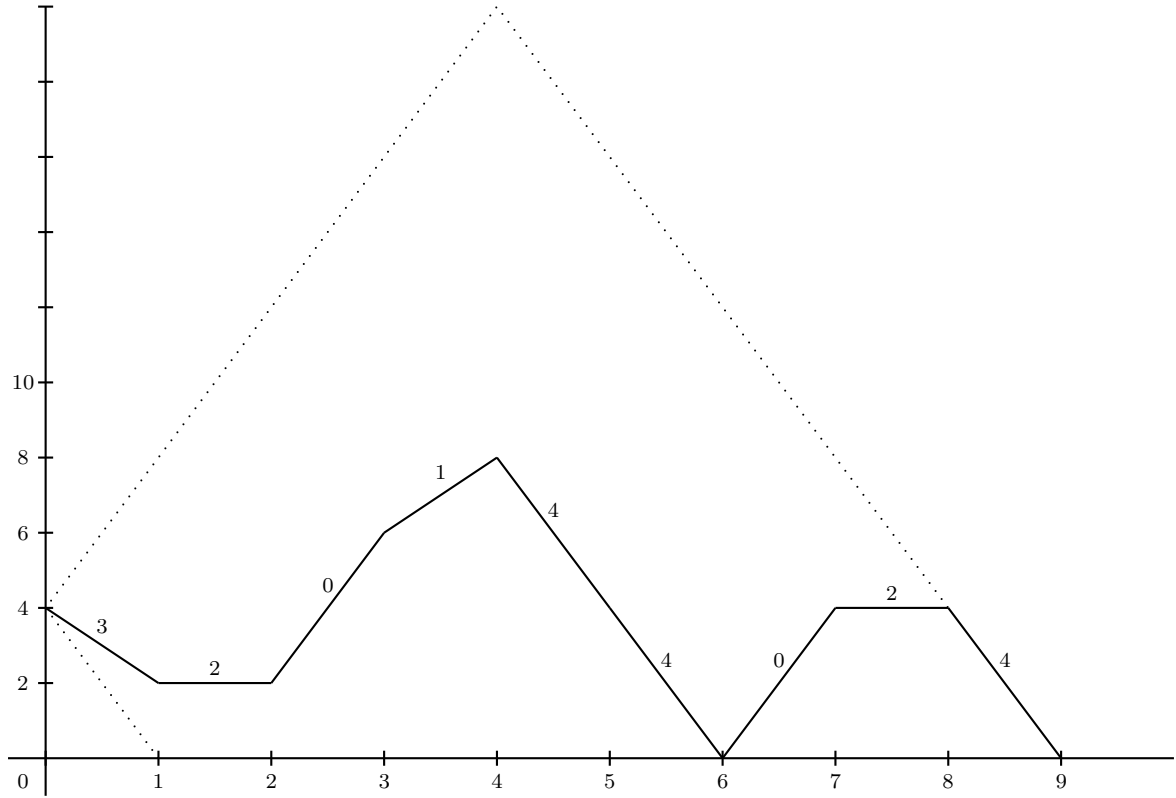
\end{center}

Keeping the above notation in mind, we go back to the consideration of the moments of the tetilla law.
In the rest of this section, we fix $f\in L^2(\R_+^2)$ as being equal to \[
f=\frac{1}{\sqrt{2}}\lp 1_{[0,1]}
\otimes 1_{[1,2]}+1_{[1,2]} \otimes 1_{[0,1]}\rp.
\]
Remember that the odd moments are all equal to zero. So, from now on, we fix an {\it even}
integer $l=2m\geq 4$.
According to Corollary \ref{cor:for-mo} and recalling the notation (\ref{nota-contr}), the $l$-th moment of $I_2(f)$ is given by
\[
S_{f,l}(f):=E\lc I_2(f)^l \rc =\sum_{\path_{l-1} \in \ca_{2,2,l}} \mathcal{C}_{f,\path_{l-1}}(f), \quad l\geq 2.
\]
More generally, we associate with every even integer $p\geq 2$ and every function $g\in L^2(\R_+^p)$ the quantity
$$S_{f,l}(g):=\sum_{\path_{l-1} \in \ca_{2,p,l}} \mathcal{C}_{f,\path_{l-1}}(g).$$
The set $\ca_{2,p,l}$ (with $p$ an even integer) is easy to visualize thanks to our walk representation. Indeed, the conditions $(i)$-$(iii)$ reduce here to: $(i)'$ $M_0=p$ and $M_{l-1}=0$; $(ii)'$
$M_{k+1}-M_k \in \{-2,0,2\}$; $(iii)'$ if $M_k=0$, then $M_{k+1}=2$, that is, the walk bounces back to $2$ each time it reaches $0$.

\smallskip

\noindent
With this representation in mind and because  $\| f\|_2=1$, it is readily checked that
\begin{eqnarray}
S_{f,2m}(f)&=& S_{f,2m-1}(1)+S_{f,2m-1}(f\pt{1} f)+S_{f,2m-1}(f \otimes f) \nonumber\\
&=&S_{f,2m-2}(f)+S_{f,2m-1}(f\pt{1} f)+S_{f,2m-1}(f \otimes f).\label{alg-1}
\end{eqnarray}
Now, observe the two relations $(f\pt{1} f) \pt{1} f =\frac{1}{2}f$ and $\lla f\pt{1} f,f \rra=0$, which give rise to the formula
\begin{equation}\label{alg-2}
S_{f,2m-1}(f \pt{1} f)=\frac{1}{2} S_{f,2m-2}(f)+S_{f,2m-2}((f\pt{1} f) \otimes f).
\end{equation}
As a result, it remains to compute $S_{f,2m-1}(f \otimes f)$ and $S_{f,2m-2} ((f\pt{1} f) \otimes f)$. To this end,
let us
introduce, for every integer $k\geq 2$, the subset $\ca_{2,p,k}^+\subset\ca_{2,p,k}$ of strictly positive walks
(up to the time $k-1$), and set
$$S^+_{f,k}(g):=\sum_{\path_{k-1} \in \ca^+_{2,p,k}} \mathcal{C}_{f,\path_{l-1}}(g).$$
In the latter formula, $\ca^+_{2,p,k}$ is of course understood as a subset of $\{0,1,2\}^{k-1}$ via the equivalence
given by (\ref{bijection}). Owing to the constraints $(i)'$-$(ii)'$, it is easily seen that, for every $\path_{2m-2} \in
\ca_{2,4,2m-1}$, there exists a smallest integer $k\in \{2,\ldots,2m-3\}$ such that the iterated contraction
$\mathcal{C}_{f,\path_{2m-2}} (f\otimes f)$ can be (uniquely) splitted into
$$\mathcal{C}_{f,\path_{2m-2}} (f\otimes f)=\mathcal{C}_{f,\path'_{k-1}}(f) \mathcal{C}_{f,\path_{2m-2-k}''}(f)$$
for some $\path_k' \in \ca_{2,2,k}^+$ and $\path_{2m-1-k}'' \in \ca_{2,2,2m-1-k}$ (see Figure 2 for an illustration). Together with the identity $\lla f\pt{1} f,f \rra=0$, this observation leads to the formula
\begin{equation}\label{alg-3}
S_{f,2m-1}(f\otimes f)=\sum_{k=1}^{m-1} S^+_{f,2k}(f) S_{f,2m-2k}(f).
\end{equation}
By a similar argument we get
\begin{equation}\label{alg-4}
S_{f,2m-2}((f\pt{1} f) \otimes f)=\sum_{k=1}^{m-2} S_{f,2k}^+(f) S_{f,2m-2k-1}(f \pt{1} f).
\end{equation}

\begin{center}
\begin{figure}[!ht]\label{pic:splitting}

\begin{pspicture}(0,0)(15,6)

%\psgrid

\psline(0,-0.5)(0,6)

\psline(-0.5,0)(15,0)

\psline(-0.1,1)(0.1,1) \psline(-0.1,2)(0.1,2) \psline(-0.1,3)(0.1,3) \psline(-0.1,4)(0.1,4) \psline(-0.1,5)(0.1,5)

\rput(-0.3,1){\tiny{$2$}} \rput(-0.3,2){\tiny{$4$}} \rput(-0.3,3){\tiny{$6$}} \rput(-0.3,4){\tiny{$8$}} \rput(-0.3,5){\tiny{$10$}}

\psline(0,-0.1)(0,0.1) \psline(1.5,-0.1)(1.5,0.1) \psline(3,-0.1)(3,0.1) \psline(4.5,-0.1)(4.5,0.1) \psline(6,-0.1)(6,0.1) \psline(7.5,-0.1)(7.5,0.1) \psline(9,-0.1)(9,0.1) \psline(10.5,-0.1)(10.5,0.1) \psline(12,-0.1)(12,0.1) \psline(13.5,-0.1)(13.5,0.1)

\rput(-0.3,-0.3){\tiny{$0$}} \rput(1.5,-0.3){\tiny{$1$}} \rput(3,-0.3){\tiny{$2$}} \rput(4.5,-0.3){\tiny{$3$}} \rput(6,-0.3){\tiny{$4$}} \rput(7.5,-0.3){\tiny{$5$}} \rput(9,-0.3){\tiny{$6$}} \rput(10.5,-0.3){\tiny{$7$}} \rput(12,-0.3){\tiny{$8$}}

\psline[linecolor=blue](0,2)(1.5,3) \psline[linecolor=blue](1.5,3)(3,3) \psline[linecolor=blue](3,3)(4.5,2) \psline[linecolor=blue](4.5,2)(6,2) \psline[linecolor=blue](6,2)(7.5,1) \psline[linecolor=red](7.5,1)(9,0) \psline[linecolor=red](9,0)(10.5,1) \psline[linecolor=red](10.5,1)(12,0)

\psline[linestyle=dotted](1.5,3)(4.5,5) \psline[linestyle=dotted](4.5,5)(10.5,1)

\rput(0.75,1.5){\tiny{$f\otimes f$}} \psline{->}(0.5,1.75)(0.1,1.9)

\end{pspicture}

\

\caption{Splitting up $\mathcal{C}_{f,\path_8}(f\otimes f)$. The blue part of the walk gives birth to $\mathcal{C}_{f,\path_5'}(f)$ with $\path_5'\in \ca^+_{2,2,6}$, while the red part corresponds to $\mathcal{C}_{f,\path_3''}(f)$ with $\path_3'' \in \ca_{2,2,4}$.}
\end{figure}
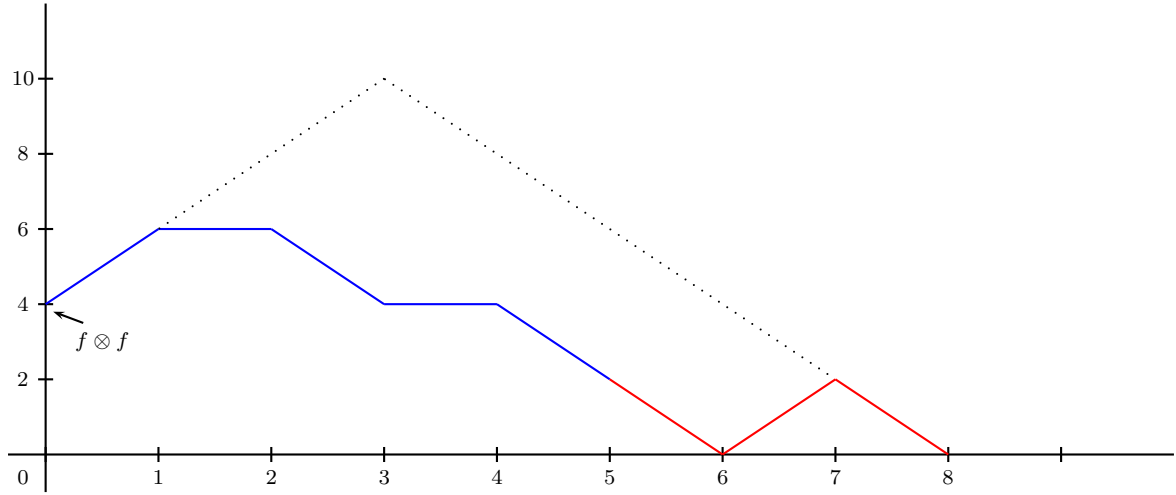
\end{center}

Going back to (\ref{alg-1}), the three formulae (\ref{alg-2}), (\ref{alg-3}) and (\ref{alg-4}) provide us with an iterative
algorithm for the computation of $S_{f,2m}(f)$. The only unknown quantities are the sums
$S_{f,2k}^+(f)$ for $k\in \{1,\ldots,m\}$. However, it turns out that the above reasoning can also be applied to
$S_{f,2k}^+(f)$ (instead of $S_{f,2k}(f)$) to give the self-contained iterative procedure:
$$\left\lbrace
\begin{array}{rcl}
S^+_{f,2m}(f) &=& S^+_{f,2m-1}(f\pt{1} f)+\sum_{k=1}^{m-1} S^+_{f,2k}(f) S^+_{f,2m-2k}(f),\\
S^+_{f,2m-1}(f \pt{1} f)&=&\frac{1}{2} S^+_{f,2m-2}(f)+\sum_{k=1}^{m-2} S^+_{f,2k}(f) S^+_{f,2m-2k-1}(f \pt{1} f).
\end{array}
\right.$$

Finally, we have proved the following result:
\begin{prop}\label{prop:algo-moments}
The even moments $E\lc \teti^{2m}\rc=S_{f,2m}(f)$ of the tetilla law are governed by the iterative algorithm:
$$\left\lbrace
\begin{array}{rcl}
S_{f,2m}(f) &=& S_{f,2m-2}(f)+S_{f,2m-1}(f\pt{1} f)+\sum_{k=1}^{m-1} S^+_{f,2k}(f) S_{f,2m-2k}(f),\\
S_{f,2m-1}(f \pt{1} f)&=&\frac{1}{2} S_{f,2m-2}(f)+\sum_{k=1}^{m-2} S^+_{f,2k}(f) S_{f,2m-2k-1}(f \pt{1} f),
\end{array}
\right.$$
with initial conditions
$$S_{f,2}(f)=1,  \quad S_{f,3}(f\pt{1} f)=\frac{1}{2}, \quad S_{f,2}^+(f)=1 \quad
\mbox{and}\quad S_{f,3}^+(f\pt{1} f)=\frac{1}{2}.$$
\end{prop}

\section{Proof of the main results}

For the sake of conciseness, we introduce the notation $h_n \approx g_n$ (for sequences of functions $h_n,g_n \in L^2(\R_+^p)$, $p\geq 0$) to indicate
that $h_n-g_n \to 0$ as $n$ tends to infinity.

\subsection{Sketch of the proof}

Before going into the technical details of the proof of Theorem \ref{mainthm}, let us try to give an intuitive
idea of the main lines of our reasoning.
To this end, let us first go back to the semicircular case, that is, to the Fourth Moment Theorem for Wigner integrals, which was
first established in \cite{knps} and then re-examined in \cite{yetanother}. Specifically, let $(f_n)\subset L^2(\R_+^q)$ be a
sequence of mirror-symmetric normalized functions; in this case, if $E\lc I_q(f_n)^4 \rc \to 2=E\lc I_1(1_{[0,1]})^2\rc $ as $n$ tends to infinity,
then $I_q(f_n)$ converges in distribution to $I_1(1_{[0,1]})$, that is, to the semicircular law. In the two afore-mentionned references, it appears clearly that the arguments leading to this convergence criterion can be organized around two successive steps:

\

\noindent
\emph{Step 1}. We observe that the convergence of the fourth moment towards $2$ forces the asymptotic behaviour of the (non-trivial) contractions of $f_n$. Indeed, from the general formula
\begin{equation}\label{starting-semi}
E\lc I_q(f_n)^4\rc=2+\sum_{r=1}^{q-1} \| f_n \pt{r} f_n \|_{2q-2r}^2,
\end{equation}
one immediately deduces that $f_n \pt{r} f_n \approx 0$ for $r\in \{1,\ldots,q-1\}$.

\noindent
\emph{Step 2}. Now the limit of $f_n \pt{r} f_n$ is known, one can make use of the formula (\ref{form-gene-moments}) to compute the limit of
$E\lc I_q(f_n)^l\rc$ for every $l$. The semicircularity of the limit is then a consequence of the following
elementary splitting:
\begin{equation}\label{elem-splitting}
E\lc I_q(f_n)^l \rc =\sum_{\substack{ \path_{l-1} \in \mathcal{A}_{q,l}\\ \path_{l-1} \in \{0,q\}^l }}((\ldots(f\pt{r_1} f) \pt{r_2} f)  \ldots \pt{r_{l-1}} f +\sum_{\substack{ \path_{l-1} \in \mathcal{A}_{q,l}\\ \path_{l-1} \notin \{0,q\}^l }} ((\ldots(f\pt{r_1} f) \pt{r_2} f)\ldots \pt{r_{l-1}} f.
\end{equation}
Indeed, thanks to \emph{Step 1} and Cauchy-Schwarz, it is easy to see that the second sum of (\ref{elem-splitting}) tends to $0$,
as each string of contractions therein involves at least one single contraction of the type
$f_n \pt{r} f_n$ with $r\in \{1,\ldots,q-1\}$. To conclude the proof, it remains to notice that, since $\| f_n \|_q=1$, one has
\begin{eqnarray*}
\sum_{\substack{ \path_{l-1} \in \mathcal{A}_{q,l}\\ \path_{l-1} \in \{0,q\}^l }}((\ldots(f\pt{r_1} f) \pt{r_2} f)  \ldots \pt{r_{l-1}} f& = &\sum_{\path_{l-1} \in \mathcal{A}_{2,l}} ((\ldots(1_{[0,1]}\pt{r_1} 1_{[0,1]}) \pt{r_2} 1_{[0,1]})  \ldots \pt{r_{l-1}} 1_{[0,1]}\\
&=& E\lc I_1(1_{[0,1]})^l \rc.
\end{eqnarray*}

\

In \cite{np-poisson}, this two-step procedure has been adapted to the case where the limit is free Poisson distributed,
represented e.g. by the integral $I_2(1_{[0,1]} \otimes 1_{[0,1]})$. In this situation, as a substitute to (\ref{starting-semi}), we
may use the more sophisticated starting relation
\begin{equation}\label{start-poisson}
E\left[ \left(I_q(f_n)^2- I_q(f_n)\right)^2 \right]=2+\| f_n \pt{q/2} f_n -f_n \|^2_q+\sum_{r\in \{1,\ldots q-1\} \backslash \{\frac{q}{2} \}} \|f_n \pt{r} f_n \|_{2q-2r}^2,
\end{equation}
and then notice that, whenever both the third and four moments converge,
then
\begin{equation}\label{start-poisson2}
E \lc \left(I_q(f_n)^2- I_q(f_n)\right)^2 \rc \to E \lc \left(I_2(1_{[0,1]} \otimes 1_{[0,1]})^2- I_2(1_{[0,1]}
\otimes 1_{[0,1]})\right)^2
\rc=2,
\end{equation}
so that, combining (\ref{start-poisson}) with (\ref{start-poisson2}), one deduces $f_n \pt{r} f_n \approx 0$ for $r\in \{1,\ldots q-1\}
\backslash \{\frac{q}{2} \}$, as well as $f_n \pt{q/2} f_n \approx f_n$ (the situation where $q$ is an odd integer can be
excluded using elementary arguments).

\

In these two cases (semicircular and Poisson), the following general idea emerges from the proof: by assuming the convergence of the (four)
first moments, one can show that the contractions of $f_n$ (asymptotically) mimick the behaviour of the contractions of the reference kernel.
Thus, in the Poisson case, the relation $f_n \pt{q/2} f_n \approx f_n$ must be seen as the (asymptotic) equivalent of the property
\[
(1_{[0,1]} \otimes 1_{[0,1]}) \pt{1} (1_{[0,1]} \otimes 1_{[0,1]})=1_{[0,1]} \otimes 1_{[0,1]}
\]
characterizing the target kernel.

\smallskip

This general idea is at the core of our reasoning as well, even if the situation turns out to be more intricate in our context.
As we have seen in Subsection \ref{subsec:moments-tetilla}, the characteristic property of
$f=\frac{1}{\sqrt{2}} (1_{[0,1]}\otimes 1_{[1,2]}+1_{[1,2]} \otimes 1_{[0,1]})$
(that is, the property that allows us to compute the moments of the tetilla law)
consists in the two relations
$$(f \pt{1} f)\pt{1} f=\frac{1}{2} f \quad , \quad (f\pt{1} f) \pt{2} f=0, $$
which involve this time double contractions. Thus, the above two-step procedure must be remodelled so as to put the double contractions of
$f_n$ as the central objects of the proof, and this is precisely where the sixth moment comes into the picture
(see the general formula (\ref{six-moment})). In brief, we will adapt the previous machinery in the following way:

\

\noindent
\emph{Step 1'}. We look for a relation similar to (\ref{starting-semi}) or (\ref{start-poisson}),
that permits to compare the asymptotic behaviour of the double contractions of $f_n$
with the double contractions of $f$, when assuming the convergence of the fourth and sixth moments. This is the aim of Subsection \ref{subsec:implic-1-2} (the expected formula corresponds to (\ref{tend})), and it leads to the proof of the implication $(i) \Rightarrow (ii)$ of Theorem \ref{mainthm}.

\medskip

\noindent
\emph{Step 2'}. In Subsection \ref{subsec:implic-2-3}, we go back to the formula (\ref{form-gene-moments}) so as to exhibit the moments of
the tetilla law in the limit. The strategy is here different from the semicircular (or Poisson) case, since it cannot be reduced to a
tricky splitting such as (\ref{elem-splitting}). Instead, we show that the sequence of moments $(E\lc I_q(f_n)^l \rc)_{l\geq 1}$ is
asymptotically (with respect to $n$) governed by the algorithm described in Proposition \ref{prop:algo-moments},
from which it becomes clear that the limit is distributed according to the tetilla law.

\subsection{Proof of Theorem \ref{mainthm}, $(i) \Longrightarrow (ii)$}\label{subsec:implic-1-2}

For the sake of clarity, we divide the proof into two (similar) parts according to the parity of $q$.

\

Let us first assume that $q\geq 2$ is an {\it even} integer, and let
$(f_n)$ be a sequence of mirror-symmetric functions of $L^2(\R_+^q)$ verifying $\|f_n\|_q^2=1$ for all $n$.
In the sequel, for $k\in\{0,\ldots,\frac{3q}{2}\}$, we write $\mathcal{B}_{2k}$ to indicate the set of those
integers $r\in\{0,\ldots,q\}$ such that \[
0\leq \frac{3q}2-k-r\leq q\wedge(2q-2r)
\]
that is, the set of integers $r$ for which the double contraction
$(f_n\stackrel{r}{\frown} f_n) \stackrel{\frac{3q}2-k-r}{\frown} f_n$ is well-defined (as an element of $L^2(\R_+^{2k})$).
Observe that $\mathcal{B}_{2k}$ merely reduces to $\{0,\ldots,\frac{3q}2-k\}$ when $k\geq \frac{q}{2}$.

\smallskip

The implication $(i)  \Rightarrow(ii)$ of Theorem \ref{mainthm} can be reformulated as follows:

\begin{prop}\label{prop-main}
Assume $E[I_q(f_n)^6]\to 8.25$ and $E[I_q(f_n)^4]\to 2.5$ as $n\to\infty$.
Then, as $n\to\infty$,
\begin{enumerate}
\item[$(a)$] $(f_n\stackrel{r}{\frown} f_n) \!\!\stackrel{\frac{3q}2-k-r}{\frown} \!\!f_n\to 0,\,
k\in\left\{0,\ldots,\frac{q}2-1\right\}\cup\left\{\frac{q}2+1,\ldots,\frac{3q}2-1\right\},\, r\in \mathcal{B}_{2k} \backslash \{0,\frac{3q}{2}-k\}$;
\item[$(b)$] $-\frac12 f_n + \sum_{r=1}^{q-1} (f_n \stackrel{r}{\frown} f_n)  \stackrel{q-r}{\frown} f_n\to 0$.
\end{enumerate}
\end{prop}

The proof of Proposition \ref{prop-main} mainly relies on the following technical lemma, whose proof is postponed to the appendix.

\begin{lemma}\label{lm1}
We have
\begin{eqnarray*}
&&\sum_{k=\frac{q}2+1}^{\frac{3q}2-1} \left\|
\sum_{r\in\mathcal{B}_{2k}} (f_n\stackrel{r}{\frown} f_n) \stackrel{\frac{3q}2-k-r}{\frown} f_n
\right\|_{2k}^2 \\
&=&\sum_{k=\frac{q}2+1}^{\frac{3q}2-1}
\sum_{r\in\mathcal{B}_{2k}} \left\|(f_n\stackrel{r}{\frown} f_n) \stackrel{\frac{3q}2-k-r}{\frown} f_n
\right\|^2_{2k}
+2\sum_{k=0}^{\frac{q}2-1} \left\|
\sum_{r\in\mathcal{B}_{2k}} (f_n\stackrel{r}{\frown} f_n) \stackrel{\frac{3q}2-k-r}{\frown} f_n
\right\|^2_{2k}.
\end{eqnarray*}
\end{lemma}

Using the multiplication formula twice leads to
\[
I_q(f_n)^3 = \sum_{r=0}^q \sum_{s=0}^{(2q-2r)\wedge q} I_{3q-2r-2s}((f_n\stackrel{r}{\frown}f_n)\stackrel{s}{\frown} f_n),
\]
or equivalently, after setting $k=3q-2r-2s$,
\[
I_q(f_n)^3 = \sum_{k=0}^{\frac{3q}2} I_{2k}\left(\sum_{r\in\mathcal{B}_{2k}}(f_n\stackrel{r}{\frown} f_n)\stackrel{\frac{3q}2-k-r}{\frown} f_n\right).
\]
We deduce, using moreover that $\|f_n\|_q^2=f_n\stackrel{q}{\frown}f_n =1$, that
\begin{eqnarray*}
I_q(f_n)^3 -\frac52 I_q(f_n)
&=& \sum_{k=0}^{\frac{q}2-1} I_{2k}\left(\sum_{r\in\mathcal{B}_{2k}}(f_n\stackrel{r}{\frown} f_n)\stackrel{\frac{3q}2-k-r}{\frown} f_n\right)
+I_q\left(-\frac12f_n+\sum_{r=1}^{q-1}(f_n \stackrel{r}{\frown} f_n)  \stackrel{q-r}{\frown} f_n \right)\\
&& +\sum_{k=\frac{q}2+1}^{\frac{3q}{2}-1} I_{2k}\left(\sum_{r\in\mathcal{B}_{2k}}(f_n\stackrel{r}{\frown} f_n)\stackrel{\frac{3q}2-k-r}{\frown} f_n\right)
+ I_{3q}(f_n\otimes f_n\otimes f_n).
\end{eqnarray*}
Hence
\begin{eqnarray}
E\big[(I_q(f_n)^3 -\frac52 I_q(f_n))^2\big]
&=& \sum_{k=0}^{\frac{q}2-1} \left\|\sum_{r\in\mathcal{B}_{2k}}(f_n\stackrel{r}{\frown} f_n)\stackrel{\frac{3q}2-k-r}{\frown} f_n\right\|^2_{2k}
+\left\|-\frac12f_n+\sum_{r=1}^{q-1}(f_n \stackrel{r}{\frown} f_n)  \stackrel{q-r}{\frown} f_n \right\|^2_q \nonumber\\
&& +\sum_{k=\frac{q}2+1}^{\frac{3q}{2}-1} \left\|\sum_{r\in\mathcal{B}_{2k}}(f_n\stackrel{r}{\frown} f_n)\stackrel{\frac{3q}2-k-r}{\frown} f_n\right\|^2_{2k}
+ 1,\label{six-moment}
\end{eqnarray}
implying in turn, thanks to Lemma \ref{lm1},
\begin{eqnarray}
&&E\big[(I_q(f_n)^3 -\frac52 I_q(f_n))^2\big]- 1 - 2\sum_{r=1}^{q-1} \| f_n\stackrel{r}{\frown} f_n \|^2_{2q-2r}\notag\\
&=& 3\sum_{k=0}^{\frac{q}2-1} \left\|\sum_{r\in\mathcal{B}_{2k}}(f_n\stackrel{r}{\frown} f_n)\stackrel{\frac{3q}2-k-r}{\frown} f_n\right\|^2_{2k}
+\left\|-\frac12f_n+\sum_{r=1}^{q-1}(f_n \stackrel{r}{\frown} f_n)  \stackrel{q-r}{\frown} f_n \right\|^2_q\notag\\
&& +\lc \sum_{k=\frac{q}2+1}^{\frac{3q}{2}-1} \sum_{r\in\mathcal{B}_{2k}}
\left\|(f_n\stackrel{r}{\frown} f_n)\stackrel{\frac{3q}2-k-r}{\frown} f_n\right\|^2_{2k}
- 2\sum_{r=1}^{q-1} \| f_n\stackrel{r}{\frown} f_n \|^2_{2q-2r}\rc\notag\\
&=& 3\sum_{k=0}^{\frac{q}2-1} \left\|\sum_{r\in\mathcal{B}_{2k}}(f_n\stackrel{r}{\frown} f_n)\stackrel{\frac{3q}2-k-r}{\frown} f_n\right\|^2_{2k}
+\left\|-\frac12f_n+\sum_{r=1}^{q-1}(f_n \stackrel{r}{\frown} f_n)  \stackrel{q-r}{\frown} f_n \right\|^2_q\notag\\
&& +\sum_{k=\frac{q}2+1}^{\frac{3q}{2}-1}
\sum_{\stackrel{r\in \mathcal{B}_{2k}}{\stackrel{r\neq \frac{3q}2-k}{r\neq 0}}}
\left\|(f_n\stackrel{r}{\frown} f_n)\stackrel{\frac{3q}2-k-r}{\frown} f_n\right\|^2_{2k}.\label{tend}
\end{eqnarray}

On the other hand, the multiplication formula, together with $\|f_n\|^2_q=f_n\stackrel{q}{\frown}f_n =1$, yields
\[
I_q(f_n)^2 = I_{2q}(f_n\otimes f_n) + 1+\sum_{r=1}^{q-1} I_{2q-2r}(f_n\stackrel{r}{\frown}f_n),
\]
so that \[
E[I_q(f_n)^4]=2+\sum_{r=1}^{q-1} \| f_n\stackrel{r}{\frown} f_n \|^2_{2q-2r}.
\]
This latter fact, combined with $E[I_q(f_n)^6]\to 8.25$ and $E[I_q(f_n)^4]\to 2.5$ as $n\to\infty$, leads to
\begin{equation}\label{tendverszero}
E\left[\left(I_q(f_n)^3 -\frac52 I_q(f_n)\right)^2\right]- 1 - 2\sum_{r=1}^{q-1} \| f_n\stackrel{r}{\frown} f_n \|_{2q-2r}^2\to 0
\quad\mbox{as $n\to\infty$.}
\end{equation}
By comparing (\ref{tendverszero}) with (\ref{tend}), we get, as $n\to\infty$,
that
\begin{equation}\label{tendverszero1}
(f_n\stackrel{r}{\frown} f_n) \stackrel{\frac{3q}2-k-r}{\frown} f_n\to 0,\quad k\in\left\{\frac{q}2+1,\ldots,\frac{3q}2-1\right\},\,
r\in \mathcal{B}_{2k} \backslash \{0,\frac{3q}{2}-k\}
\end{equation}
and that
\begin{equation}\label{tendverszero3}
-\frac12 f_n + \sum_{r=1}^{q-1} (f_n \stackrel{r}{\frown} f_n)  \stackrel{q-r}{\frown} f_n\to 0.
\end{equation}
Now, for $k\in\{0,\ldots,\frac{q}2-1\}$ and $r\in\mathcal{B}_{2k}$,
we have, thanks to the third point of Proposition \ref{contr-link} (see the appendix),
\begin{eqnarray}
\left\|(f_n\stackrel{r}{\frown} f_n)\stackrel{\frac{3q}2-k-r}{\frown} f_n\right\|^2_{2k}
&=&\left\langle
(f_n\stackrel{q-r}{\frown} f_n)\stackrel{r+k-\frac{q}2}{\frown} f_n,
(f_n\stackrel{\frac{q}2+k-r}{\frown} f_n)\stackrel{r}{\frown} f_n
\right\rangle_{2q-2k}\notag\\
&\leq&
\|f_n\|_q^3\left\|
(f_n\stackrel{\frac{q}2+k-r}{\frown} f_n)\stackrel{r}{\frown} f_n
\right\|_{2q-2k}\notag\\
&=& \left\|
(f_n\stackrel{\frac{q}2+k-r}{\frown} f_n)\stackrel{r}{\frown} f_n
\right\|_{2q-2k}
= \left\|
(f_n\stackrel{s}{\frown} f_n)\stackrel{\frac{3q}2-l-s}{\frown} f_n
\right\|_{2l},\notag\\
\label{tendverszero2}
\end{eqnarray}
with $l=q-k\in\left\{\frac{q}2+1,\ldots,q\right\}$ and $s=\frac{3q}2-l-r\in\mathcal{B}_{2l}$.
Hence, (\ref{tendverszero1}) with (\ref{tendverszero2}) imply together that
\begin{equation}\label{tendverszero4}
(f_n\stackrel{r}{\frown} f_n) \stackrel{\frac{3q}2-k-r}{\frown} f_n\to 0,\quad k\in\left\{0,\ldots,\frac{q}2-1\right\},\,r\in\mathcal{B}_{2k},
\end{equation}
and conclude the proof of Proposition \ref{prop-main}, see indeed (\ref{tendverszero1}), (\ref{tendverszero3}) and (\ref{tendverszero4}).

\fin

The proof when $q\geq 3$ is {\it odd} is similar. In this situation, we set $p:=q-1$ and,
for $k\in \{0,\ldots, \frac{3p}{2}\}$, we denote by $\mathcal{B}_{2k+1}$ the set of those integers $r\in \{0,\ldots q\}$ for
which the double contraction $(f_n \pt{r} f_n) \pt{\frac{3p}{2}+1-k-r} f_n$ is well-defined (as an element of
$L^2(\R_+^{2k+1})$). The desired conclusion can then be reformulated in the following way:
\begin{prop}\label{prop-main-2}
Assume $E[I_q(f_n)^6]\to 8.25$ and $E[I_q(f_n)^4]\to 2.5$ as $n\to\infty$.
Then, as $n\to\infty$,
\begin{enumerate}
\item[$(a)$] $(f_n\stackrel{r}{\frown} f_n) \stackrel{\frac{3p}2+1-k-r}{\frown} f_n\to 0,\quad
k\in\left\{0,\ldots,\frac{3p}2\right\}\backslash \{\frac{p}{2}\},\, r\in\mathcal{B}_{2k+1}\backslash \{0,\frac{3p}2+1-k\}$;
\item[$(b)$] $-\frac12 f_n + \sum_{r=1}^{q-1} (f_n \stackrel{r}{\frown} f_n)  \stackrel{q-r}{\frown} f_n\to 0$.
\end{enumerate}
\end{prop}

With the same arguments as in the proof of Lemma \ref{lm1}, we get the following analogous identity
as a starting point towards the proof of Proposition \ref{prop-main-2}:
\begin{lemma}\label{lm1-1}
We have
\begin{eqnarray*}
&&\sum_{k=\frac{p}2+1}^{\frac{3p}2} \left\|
\sum_{r\in\mathcal{B}_{2k+1}} (f_n\stackrel{r}{\frown} f_n) \stackrel{\frac{3p}2+1-k-r}{\frown} f_n
\right\|^2 \\
&=&\sum_{k=\frac{p}2+1}^{\frac{3p}2}
\sum_{r\in\mathcal{B}_{2k+1}} \left\|(f_n\stackrel{r}{\frown} f_n) \stackrel{\frac{3p}2+1-k-r}{\frown} f_n
\right\|^2
+2\sum_{k=0}^{\frac{p}2-1} \left\|
\sum_{r\in\mathcal{B}_{2k+1}} (f_n\stackrel{r}{\frown} f_n) \stackrel{\frac{3p}2+1-k-r}{\frown} f_n
\right\|^2.
\end{eqnarray*}
\end{lemma}

Since we have
\begin{eqnarray*}
\lefteqn{E\big[(I_q(f_n)^3 -\frac52 I_q(f_n))^2\big]=}\\
&& \sum_{k=0}^{\frac{p}2-1} \left\|\sum_{r\in\mathcal{B}_{2k+1}}(f_n\stackrel{r}{\frown} f_n)\stackrel{\frac{3p}2+1-k-r}{\frown} f_n\right\|^2
+\left\|-\frac12f_n+\sum_{r=1}^{q-1}(f_n \stackrel{r}{\frown} f_n)  \stackrel{q-r}{\frown} f_n \right\|^2\\
&& +\sum_{k=\frac{p}2+1}^{\frac{3p}{2}} \left\|\sum_{r\in\mathcal{B}_{2k}}(f_n\stackrel{r}{\frown} f_n)\stackrel{\frac{3p}2+1-k-r}{\frown} f_n\right\|^2
+ 1,
\end{eqnarray*}
the conclusion is now easily derived by means of the same arguments as in the even case.

\subsection{Proof of Theorem \ref{mainthm}, $(ii) \Longrightarrow (iii)$}\label{subsec:implic-2-3}

The proof of $(ii) \Rightarrow (iii)$ will make use, among other things, of the following readily-checked identity:
\begin{lemma}\label{lem:forbid-walks}
Let $q\geq 2$ be an integer and $f\in L^2(\R_+^q)$. If $r\in \{1,\ldots, q-1\}$ and $r' \in \{1,\ldots,q\}$
are such that $r'+2r >2q$ then, for every integer $p\geq 1$ and every function $g\in L^2(\R_+^{p})$,
$$((g \otimes f) \pt{r} f) \pt{r'} f =g \pt{r'+2r-2q} ((f \pt{r} f) \pt{2q-2r} f).$$
In particular, for any sequence $(f_n)$ in $L^2(\R_+^q)$ satisfying Condition
(\ref{contract-1}) of Theorem \ref{mainthm}, one has $((g\otimes f_n) \pt{r} f_n) \pt{r'} f_n \approx 0$ (remember that the notation $\approx$ has been introduced at the beginning of the section).
\end{lemma}

\begin{remark}\label{rk:forbid-walks}
{\rm
The previous result should be understood as follows: if the double contraction $(r,r')$ of $g\otimes f_n$ interacts with the arguments of $g$, that is, if $((g\otimes f_n) \pt{r} f_n) \pt{r'} f_n \neq g \otimes ((f_n \pt{r} f_n) \pt{r'} f_n)$, then it tends to $0$
(see also Figure 3).
}
\end{remark}

\begin{center}
\begin{figure}[!ht]\label{demo-q-4}

\begin{pspicture}(0,0)(15,10)
\psline(-0.5,0)(15,0)

\psline(0,-0.1)(0,0.1) \psline(1.5,-0.1)(1.5,0.1) \psline(3,-0.1)(3,0.1) \psline(4.5,-0.1)(4.5,0.1) \psline(6,-0.1)(6,0.1) \psline(7.5,-0.1)(7.5,0.1) \psline(9,-0.1)(9,0.1) \psline(10.5,-0.1)(10.5,0.1) \psline(12,-0.1)(12,0.1) \psline(13.5,-0.1)(13.5,0.1)

\psline(0,2)(1.5,4) \psline(0,2)(1.5,0) \psline(0,2)(1.5,1) \psline(0,2)(1.5,2) \psline(0,2)(1.5,3)
\psline(1.5,4)(3,2) \psline(1.5,4)(3,3) \psline(1.5,4)(3,4) \psline(1.5,4)(3,5) \psline[linestyle=dotted](1.5,4)(3,6)
\psline(3,5)(4.5,4) \psline(3,4)(4.5,4) \psline(3,3)(4.5,4) \psline[linestyle=dotted](3,5)(4.5,7) \psline[linestyle=dotted](3,4)(4.5,6) \psline[linestyle=dotted](3,3)(4.5,5)
\psline(4.5,4)(6,5) \psline(4.5,4)(6,4) \psline(4.5,4)(6,3) \psline(4.5,4)(6,2) \psline[linestyle=dotted](4.5,4)(6,6)
\psline(6,5)(7.5,4) \psline(6,4)(7.5,4) \psline(6,3)(7.5,4) \psline[linestyle=dotted](6,5)(7.5,7) \psline[linestyle=dotted](6,4)(7.5,6) \psline[linestyle=dotted](6,3)(7.5,5)
\psline(7.5,4)(9,2)

\psline[linecolor=red](3,3)(4.5,2)  \psline[linecolor=red](3,3)(4.5,1)
\psline[linecolor=red](6,3)(7.5,2)  \psline[linecolor=red](6,3)(7.5,1)

\psline(1.5,0)(3,2) \psline[linestyle=dotted](1.5,1)(2,1.5) \psline[linestyle=dotted](1.5,1)(2,1.25) \psline[linestyle=dotted](1.5,2)(2,2) \psline[linestyle=dotted](1.5,2)(2,2.5) \psline[linestyle=dotted](1.5,3)(2,2.75) \psline[linestyle=dotted](1.5,3)(2,3.5)

\psline[linestyle=dotted](3,2)(3.5,2.5) \psline[linestyle=dotted](3,2)(3.5,2.25) \psline[linestyle=dotted](3,2)(3.5,2) \psline[linestyle=dotted](3,2)(3.5,1.75) \psline[linestyle=dotted](3,2)(3.5,1.5)

\psline[linestyle=dotted](6,2)(6.5,2.5) \psline[linestyle=dotted](6,2)(6.5,2.25) \psline[linestyle=dotted](6,2)(6.5,2) \psline[linestyle=dotted](6,2)(6.5,1.75) \psline[linestyle=dotted](6,2)(6.5,1.5)

\psline[linestyle=dotted](9,2)(9.5,2.5) \psline[linestyle=dotted](9,2)(9.5,2.25) \psline[linestyle=dotted](9,2)(9.5,2) \psline[linestyle=dotted](9,2)(9.5,1.75) \psline[linestyle=dotted](9,2)(9.5,1.5)

\psline[linestyle=dotted](7.5,4)(8,4.5) \psline[linestyle=dotted](7.5,4)(8,4.25) \psline[linestyle=dotted](7.5,4)(8,4) \psline[linestyle=dotted](7.5,4)(8,3.75)

\rput(-0.3,2){$f_n$}

\rput(0,-0.3){\tiny{$0$}} \rput(3,-0.3){\tiny{$2$}} \rput(6,-0.3){\tiny{$4$}} \rput(9,-0.3){\tiny{$6$}}
\end{pspicture}

\

\caption{($q=4$) According to Lemma \ref{lem:forbid-walks}, the sum of (the iterated contractions represented by) those walks which run on a red line tends to $0$.}
\end{figure}
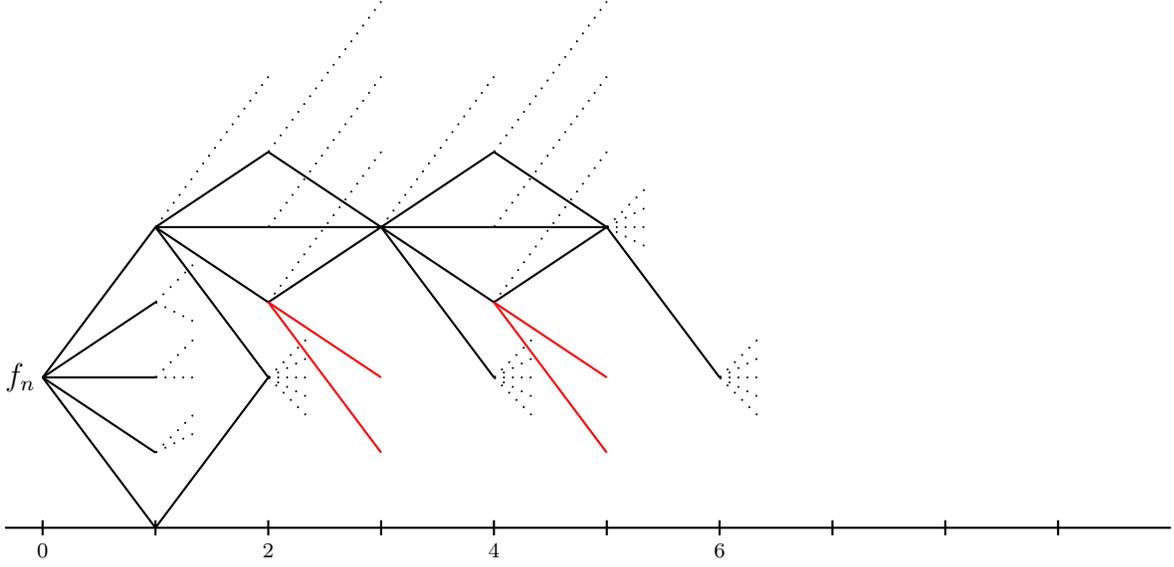
\end{center}

Now, fix an integer $q\geq 2$ and consider a sequence $(f_n)$ in $L^2(\R_+^q)$ such that both
$\|f_n \|_q=1$ and conditions
(\ref{contract-1}) and (\ref{contract-2}) of Theorem \ref{mainthm} are satisfied. To prove that $I_q(f_n)$ converges to the
tetilla law, we need to establish that the sequence $(E\lc I_q(f_n)^l\rc)_{l\geq 2}=(S_{f_n,l}(f_n))_{l\geq 2}$ of its
moments is asymptotically (with respect to $n$) governed by the algorithm of Proposition \ref{prop:algo-moments}.
(We recall that the general notation $S_{f,l}(g)$ has been introduced in Subsection \ref{subsec:moments-tetilla}.)

Let us first consider the even moments, and fix $l=2m \geq 2$. Since $\|f_n\|_q=1$, we have, as in (\ref{alg-1}),
\begin{eqnarray}
S_{f_n,2m}(f_n) &=& S_{f_n,2m-1}(1)+\sum_{r=1}^{q-1} S_{f_n,2m-1}(f_n \pt{r} f_n)+S_{f_n,2m-1}(f_n \otimes f_n) \nonumber\\
&=& S_{f_n,2m-2}(f_n)+\sum_{r=1}^{q-1} S_{f_n,2m-1}(f_n \pt{r} f_n)+S_{f_n,2m-1}(f_n \otimes f_n).
\end{eqnarray}
Then, due to the conditions (\ref{contract-1}) and (\ref{contract-2}) of Theorem \ref{mainthm}, we successively deduce
\begin{eqnarray}
\sum_{r=1}^{q-1} S_{f_n,2m-1}(f_n \pt{r} f_n) &\approx & \sum_{r=1}^{q-1} S_{f_n,2m-2}((f_n \pt{r} f_n) \pt{q-r} f_n)+\sum_{r=1}^{q-1} S_{f_n,2m-2}((f_n \pt{r} f_n) \otimes f_n) \nonumber\\
&\approx & \frac{1}{2} S_{f_n,2m-2}(f_n)+\sum_{r=1}^{q-1} S_{f_n,2m-2}((f_n \pt{r} f_n) \otimes f_n).
\end{eqnarray}
We are thus left with the computation,  for all $r\in \{1,\ldots,q-1\}$, of the sums $S_{f_n,2m-1}(f_n \otimes f_n)$ and
$S_{f_n,2m-2}((f_n \pt{r} f_n) \otimes f_n)$. To this end, notice that, thanks to the result of Lemma \ref{lem:forbid-walks}
(see also Remark \ref{rk:forbid-walks}), the splitting argument used in Subsection \ref{subsec:moments-tetilla} (see Figure 2)
can be applied in this situation as well, as the three figures 3, 4 and 5 illustrate it.
Note in particular that the splitting no longer
applies to each walk individually, but to a union of well-chosen walks (that is, a sum of well-chosen iterated contractions) that
puts forward the pattern $\sum_{r=1}^{q-1}(f_n \pt{r} f_n) \pt{q-r} f_n$, instead of $(f \pt{1} f) \pt{1} f$ as in the
two variables case.

\begin{center}
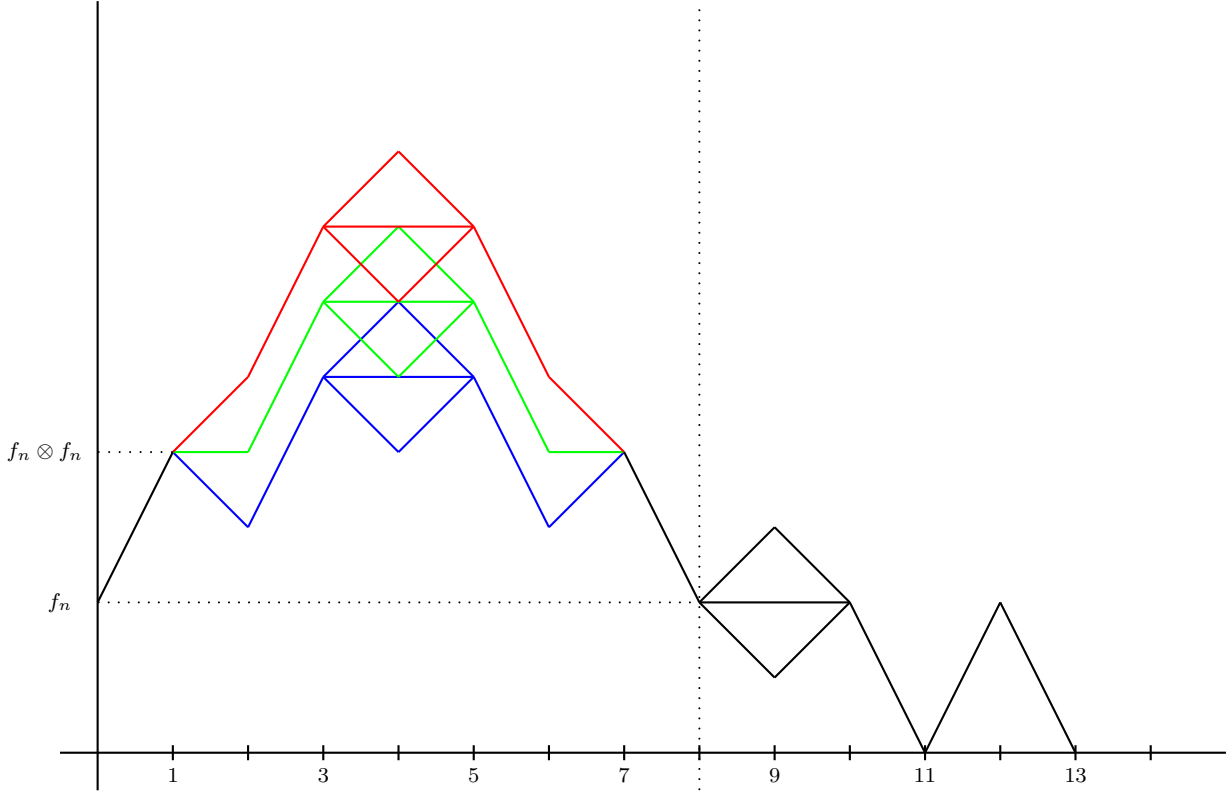
\begin{figure}[!ht]

\begin{pspicture}(0,0)(15,10)
\psline(-0.5,0)(15,0)

%\psgrid

 \psline(1,-0.1)(1,0.1) \psline(2,-0.1)(2,0.1) \psline(3,-0.1)(3,0.1) \psline(4,-0.1)(4,0.1) \psline(5,-0.1)(5,0.1) \psline(6,-0.1)(6,0.1) \psline(7,-0.1)(7,0.1) \psline(8,-0.1)(8,0.1) \psline(9,-0.1)(9,0.1) \psline(10,-0.1)(10,0.1) \psline(11,-0.1)(11,0.1) \psline(12,-0.1)(12,0.1) \psline(13,-0.1)(13,0.1) \psline(14,-0.1)(14,0.1)

\psline(0,-0.5)(0,10)

\psline(0,2)(1,4)

\psline[linecolor=blue](1,4)(2,3) \psline[linecolor=blue](2,3)(3,5) \psline[linecolor=blue](3,5)(4,4) \psline[linecolor=blue](4,4)(5,5) \psline[linecolor=blue](3,5)(5,5) \psline[linecolor=blue](3,5)(4,6) \psline[linecolor=blue](4,6)(5,5) \psline[linecolor=blue](5,5)(6,3) \psline[linecolor=blue](6,3)(7,4)

\psline[linecolor=green](1,4)(2,4) \psline[linecolor=green](2,4)(3,6) \psline[linecolor=green](3,6)(4,5) \psline[linecolor=green](4,5)(5,6) \psline[linecolor=green](3,6)(5,6) \psline[linecolor=green](3,6)(4,7) \psline[linecolor=green](4,7)(5,6) \psline[linecolor=green](5,6)(6,4) \psline[linecolor=green](6,4)(7,4)

\psline[linecolor=red](1,4)(2,5) \psline[linecolor=red](2,5)(3,7) \psline[linecolor=red](3,7)(4,6) \psline[linecolor=red](4,6)(5,7) \psline[linecolor=red](3,7)(5,7) \psline[linecolor=red](3,7)(4,8)\psline[linecolor=red](4,8)(5,7) \psline[linecolor=red](5,7)(6,5) \psline[linecolor=red](6,5)(7,4)

\psline(7,4)(8,2)

\psline[linestyle=dotted](8,-0.5)(8,10) \psline[linestyle=dotted](0,2)(8,2) \psline[linestyle=dotted](0,4)(1,4)

\psline(8,2)(9,1) \psline(9,1)(10,2) \psline(8,2)(10,2) \psline(8,2)(9,3) \psline(9,3)(10,2)

\psline(10,2)(11,0) \psline(11,0)(12,2) \psline(12,2)(13,0)

\rput(-0.7,4){\tiny{$f_n\otimes f_n$}} \rput(-0.5,2){\tiny{$f_n$}}

\rput(1,-0.3){\tiny{$1$}} \rput(3,-0.3){\tiny{$3$}} \rput(5,-0.3){\tiny{$5$}} \rput(7,-0.3){\tiny{$7$}} \rput(9,-0.3){\tiny{$9$}} \rput(11,-0.3){\tiny{$11$}} \rput(13,-0.3){\tiny{$13$}}
\end{pspicture}

\

\caption{($q=4$) A splitting (at time $8$) in $\ca_{4,8,13}$. Observe in particular how the blue, green and red (joint) walks meet together at time $7$ so as to retrieve the pattern $(f_n\pt{1} f_n) \pt{3} f_n+(f_n \pt{2} f_n) \pt{2} f_n+(f_n \pt{3} f_n) \pt{1} f_n$ that begins at time $1$.}
\end{figure}
\end{center}

\begin{center}
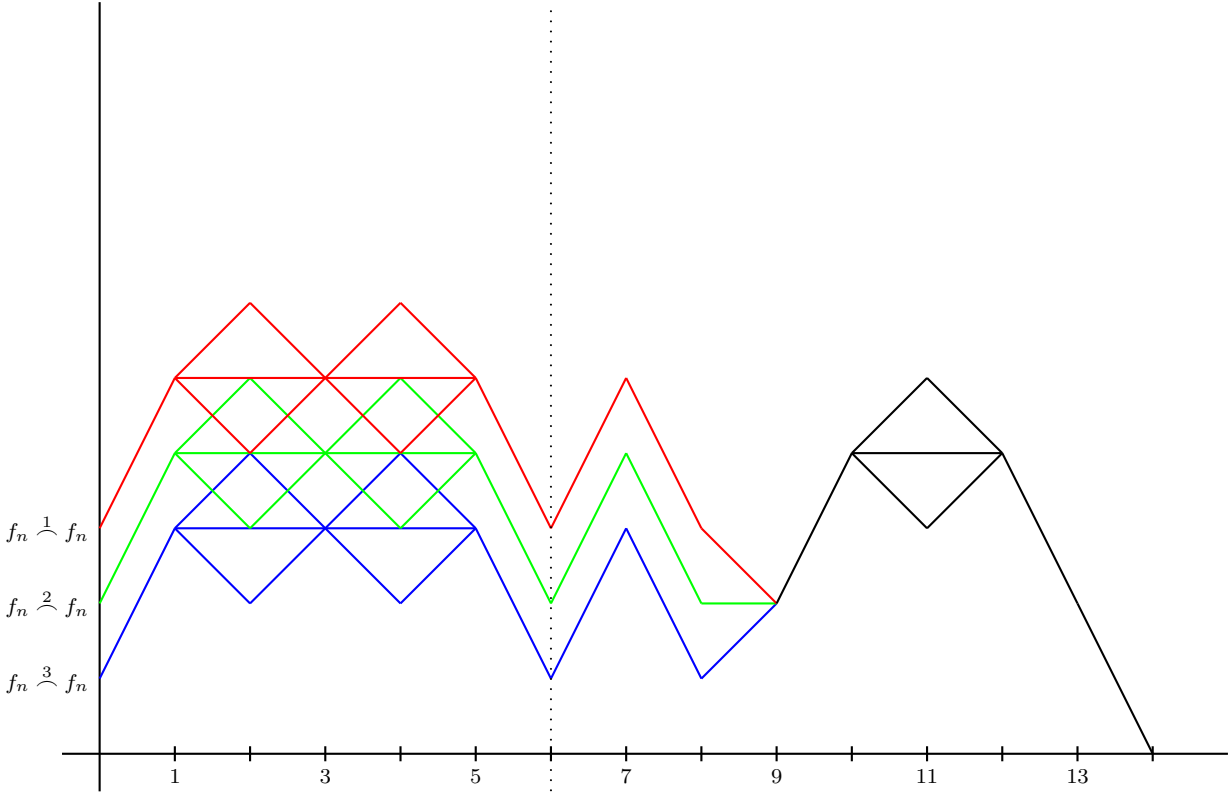
\begin{figure}[!ht]

\begin{pspicture}(0,0)(15,10)
\psline(-0.5,0)(15,0)

%\psgrid

 \psline(1,-0.1)(1,0.1) \psline(2,-0.1)(2,0.1) \psline(3,-0.1)(3,0.1) \psline(4,-0.1)(4,0.1) \psline(5,-0.1)(5,0.1) \psline(6,-0.1)(6,0.1) \psline(7,-0.1)(7,0.1) \psline(8,-0.1)(8,0.1) \psline(9,-0.1)(9,0.1) \psline(10,-0.1)(10,0.1) \psline(11,-0.1)(11,0.1) \psline(12,-0.1)(12,0.1) \psline(13,-0.1)(13,0.1) \psline(14,-0.1)(14,0.1)

\psline(0,-0.5)(0,10)

\psline[linecolor=blue](0,1)(1,3) \psline[linecolor=blue](1,3)(2,2) \psline[linecolor=blue](2,2)(3,3) \psline[linecolor=blue](1,3)(3,3) \psline[linecolor=blue](1,3)(2,4) \psline[linecolor=blue](2,4)(3,3)  \psline[linecolor=blue](3,3)(4,2) \psline[linecolor=blue](4,2)(5,3) \psline[linecolor=blue](3,3)(5,3) \psline[linecolor=blue](3,3)(4,4) \psline[linecolor=blue](4,4)(5,3) \psline[linecolor=blue](5,3)(6,1) \psline[linecolor=blue](6,1)(7,3) \psline[linecolor=blue](7,3)(8,1) \psline[linecolor=blue](8,1)(9,2)

\psline[linecolor=green](0,2)(1,4) \psline[linecolor=green](1,4)(2,3) \psline[linecolor=green](2,3)(3,4) \psline[linecolor=green](1,4)(3,4) \psline[linecolor=green](1,4)(2,5) \psline[linecolor=green](2,5)(3,4)  \psline[linecolor=green](3,4)(4,3) \psline[linecolor=green](4,3)(5,4) \psline[linecolor=green](3,4)(5,4) \psline[linecolor=green](3,4)(4,5) \psline[linecolor=green](4,5)(5,4) \psline[linecolor=green](5,4)(6,2) \psline[linecolor=green](6,2)(7,4) \psline[linecolor=green](7,4)(8,2) \psline[linecolor=green](8,2)(9,2)

\psline[linecolor=red](0,3)(1,5) \psline[linecolor=red](1,5)(2,4) \psline[linecolor=red](2,4)(3,5) \psline[linecolor=red](1,5)(3,5) \psline[linecolor=red](1,5)(2,6) \psline[linecolor=red](2,6)(3,5)  \psline[linecolor=red](3,5)(4,4) \psline[linecolor=red](4,4)(5,5) \psline[linecolor=red](3,5)(5,5) \psline[linecolor=red](3,5)(4,6) \psline[linecolor=red](4,6)(5,5) \psline[linecolor=red](5,5)(6,3) \psline[linecolor=red](6,3)(7,5) \psline[linecolor=red](7,5)(8,3) \psline[linecolor=red](8,3)(9,2)

\psline[linestyle=dotted](6,-0.5)(6,10)

\psline(9,2)(10,4) \psline(10,4)(11,3) \psline(11,3)(12,4) \psline(10,4)(12,4) \psline(10,4)(11,5) \psline(11,5)(12,4) \psline(12,4)(13,2) \psline(13,2)(14,0)

\rput(-0.7,1){\tiny{$f_n \pt{3} f_n$}} \rput(-0.7,2){\tiny{$f_n \pt{2} f_n$}} \rput(-0.7,3){\tiny{$f_n \pt{1} f_n$}}

\rput(1,-0.3){\tiny{$1$}} \rput(3,-0.3){\tiny{$3$}} \rput(5,-0.3){\tiny{$5$}} \rput(7,-0.3){\tiny{$7$}} \rput(9,-0.3){\tiny{$9$}} \rput(11,-0.3){\tiny{$11$}} \rput(13,-0.3){\tiny{$13$}}
\end{pspicture}

\

\caption{($q=4$) Three splittings (at time $6$) in $\ca_{4,2,15}$ (blue), $\ca_{4,4,15}$ (green) and $\ca_{4,6,15}$ (red), that merge in the second part.}
\end{figure}
\end{center}

\noindent
With this representation in mind, we (asymptotically) recover the formula
\begin{equation}
S_{f_n,2m-1}(f_n \otimes f_n) \approx \sum_{k=1}^{m-1} S_{f_n,2k}^+(f_n) S_{f_n,2m-2k}(f_n),
\end{equation}
as well as
\begin{equation}
S_{f_n,2m-2}((f_n \pt{r} f_n) \otimes f_n) \approx \sum_{k=1}^{m-2} S_{f_n,2k}^+(f_n) S_{f_n,2m-2k-1}(f_n \pt{r} f_n),
\end{equation}
for every $r=1,\ldots,q-1$. (We recall that the notation $S^+_{f,l}(g)$ has been defined in
Subsection \ref{subsec:moments-tetilla}.) Consequently, by setting
$$N_{f_n,l}:=\sum_{r=1}^{q-1} S_{f_n,l}(f_n \pt{r} f_n),$$
we get the asymptotic algorithm
$$\left\lbrace
\begin{array}{rcl}
S_{f_n,2m}(f_n) &\approx & S_{f_n,2m-2}(f_n)+N_{f_n,2m-1}+\sum_{k=1}^{m-1} S^+_{f_n,2k}(f_n) S_{f_n,2m-2k}(f_n),\\
N_{f_n,2m-1}&\approx &\frac{1}{2} S_{f_n,2m-2}(f_n)+\sum_{k=1}^{m-2} S^+_{f_n,2k}(f_n) N_{f,2m-2k-1}.
\end{array}
\right.$$
Then, as in Subsection \ref{subsec:moments-tetilla}, it is not hard to see that the above arguments
apply to $S^+_{f_n,2k}(f_n)$ as well, thus giving
$$\left\lbrace
\begin{array}{rcl}
S^+_{f_n,2m}(f_n) &\approx & N^+_{f_n,2m-1}+\sum_{k=1}^{m-1} S^+_{f_n,2k}(f_n) S^+_{f_n,2m-2k}(f_n),\\
N^+_{f_n,2m-1}&\approx &\frac{1}{2} S^+_{f_n,2m-2}(f_n)+\sum_{k=1}^{m-2} S^+_{f_n,2k}(f_n) N^+_{f,2m-2k-1},
\end{array}
\right.$$
where we have naturally set $N^+_{f_n,l}:=\sum_{r=1}^{q-1} S^+_{f_n,l}(f_n \pt{r} f_n)$. Together with the initial conditions
$S_{f_n,2}(f_n)=S^+_{f_n,2}(f_n)=\| f_n\|_q^2=1$ and
$N_{f_n,3}=N_{f_n,3}^+ \approx \frac{1}{2}$,
we finally recognize the iterative procedure that governs the moments of the tetilla law
(see Proposition \ref{prop:algo-moments}), which concludes the proof.
\

As far as the odd moments are concerned, we can follow the same lines of reasoning  as above, and derive an analogous iterative formula.
However, due to the assumption (\ref{contract-1}) of Theorem \ref{mainthm}, the initial condition $S_{f_n,3}(f_n)$ of the resulting algorithm tends
to $0$, implying in turn that all the odd moments of $I_q(f_n)$ asymptotically vanish, as expected.

\subsection{Proof of Corollary \ref{notetilla}}

Assume that $F = I_q(f)$, where $f$ is a mirror symmetric element of $L^2(\R_+^q)$ such that $\|f\|^2_{q}=1$.
To achieve a contradiction, we suppose
that  $(f\stackrel{1}{\frown}f)\stackrel{1}{\frown} f$ is zero almost everywhere
(which would be the case if $F$ had the tetilla law, according to Theorem \ref{mainthm}). That is,
for almost all $a,b\in\R_+$ and ${\bf r}_{q-2},{\bf s}_{q-2},{\bf t}_{q-2}\in\R_+^{q-2}$, we have
\begin{equation}\label{bravoaurelien}
\int_{\R_+^2}f(a,{\bf r}_{q-2},u)f(u,{\bf s}_{q-2},v)f(v,{\bf t}_{q-2},b)dudv = 0.
\end{equation}
Now, consider an orthonormal basis $(e_n)_{n\in \N}$ of $L^2(\R_+)$ and, for every ${\bf i}_{l}=(i_1,\ldots , i_l) \in \N^l$ ($l\in \N$), set $e_{{\bf i}_l}=e_{i_1} \otimes \ldots \otimes e_{i_l}$. Let us also introduce, for every ${\bf i}_{q-2} \in \N^{q-2}$, the function
\[
g_{{\bf i}_{q-2}}(a,b)=\int_{\R_+^{q-2}} f(a,{\bf x}_{q-2},b)e_{{\bf i}_{q-2}}({\bf x}_{q-2})d{\bf x}_{q-2};\quad a,b\in\R_+.
\]
Using (\ref{bravoaurelien}), we immediately deduce that $(g_{{\bf i}_{q-2}}
\stackrel{1}{\frown} g_{{\bf i}_{q-2}} ) \stackrel{1}{\frown} g_{{\bf i}_{q-2}} = 0$ almost everywhere.
Hence $\|g_{{\bf i}_{q-2}}
\stackrel{1}{\frown} g_{{\bf i}_{q-2}}\|^2_2 = \langle (g_{{\bf i}_{q-2}}
\stackrel{1}{\frown} g_{{\bf i}_{q-2}} ) \stackrel{1}{\frown} g_{{\bf i}_{q-2}}, g_{{\bf i}_{q-2}} \rangle_2 = 0$ so that, for every $j\in \N$,
\[
\|g_{{\bf i}_{q-2}} \stackrel{1}{\frown} e_j\|^2_1 =  \langle g_{{\bf i}_{q-2}}
\stackrel{1}{\frown} g_{{\bf i}_{q-2}} , e_j\otimes e_j\rangle_2=0.\]
As a consequence, for all $j,k\in \N$, $\int_{\R_+^2} g_{{\bf i}_{q-2}}(a,b)e_j(a)e_k(b) \, dadb=0$. Otherwise stated, one has, for every ${\bf i}_q \in \N^q$, $\int_{\R_+^q} f( {\bf x}_q) e_{{\bf i}_q}({\bf x}_q) d {\bf x}_q=0$. This proves that $f=0$ a.e. and contradicts the normalization $\| f\|_q=1$.

\section{Appendix: Proof of Lemma \ref{lm1}}

During the proof of Theorem \ref{mainthm}, we made use of Lemma \ref{lm1}, as well as the following
Proposition \ref{contr-link}. This appendix is devoted to their respective proofs.

\begin{prop}\label{contr-link}
Let $q\geq 2$ be an integer, and let $f\in L^2(\R_+^q)$.
\begin{enumerate}
\item If $r,s,r',s'$ are four positive integers satisfying
$r+s=r'+s'\leq q$, $r<r'$ and $r'+s\geq q$, then
\[
\langle (f \stackrel{r}{\frown} f) \stackrel{s}{\frown} f, (f\stackrel{r'}{\frown}f)\stackrel{s'}{\frown}f\rangle_{3q-2r-2s}
=
\langle (f \stackrel{s}{\frown} f) \stackrel{2q-2s-r}{\frown} f, (f\stackrel{q-r'}{\frown}f)\stackrel{q-s'}{\frown}f\rangle_{2r+2s-q}.
\]
\item If $r,s,r',s'$ are four positive integers satisfying
$r+s=r'+s'\leq q$, $r<r'$ and $r'+s\leq q$, then
\[
\langle (f \stackrel{r}{\frown} f) \stackrel{s}{\frown} f, (f\stackrel{r'}{\frown}f)\stackrel{s'}{\frown}f\rangle_{3q-2r-2s}
=
\langle (f \stackrel{s}{\frown} f) \stackrel{q-s'}{\frown} f, (f\stackrel{q-r'}{\frown}f)\stackrel{2r'-r}{\frown}f\rangle_{q+2r-2r'}.
\]
\item If $r$ and $s$ are two positive integers satisfying $r\leq q$, $s\leq q\wedge(2q-2r)$ and
$r+s\geq q$, then
\[
\| (f \stackrel{r}{\frown} f) \stackrel{s}{\frown} f\|_{3q-2r-2s}^2
=
\langle (f \stackrel{q-r}{\frown} f) \stackrel{q-s}{\frown} f, (f\stackrel{2q-2r-s}{\frown}f)\stackrel{r}{\frown}f\rangle_{2r+2s-q}.
\]
\end{enumerate}
\end{prop}

\begin{remark}
{\rm
The main interest of (1) is to pass from inner products involving functions
of $3q-2r-2s\geq q$ variables into inner products involving functions of $2r+2s-q\leq q$ variables only.
(A similar remark obviously holds for (2) and (3) as well.)
This nice property of double contractions is to play a major role in the proof of
Lemma \ref{lm1}.
}
\end{remark}

{\it Proof}. During all the proof, we use the following short-hand notation. For any integer $\alpha\geq 1$, we denote by
${\bf x}_\alpha$ the element of $\R_+^{\alpha}$ defined as ${\bf x}_\alpha=(x_1,\ldots,x_\alpha)$, whereas ${\bf x}_\alpha^*$ stands for
its mirror counterpart, that is,
${\bf x}_\alpha^*=(x_\alpha,\ldots,x_1)$. By extension, we allow $\alpha$ to be zero; in this case,
the implicit convention is to remove ${\bf x}_\alpha$, as well as ${\bf x}_\alpha^*$, in each expression containing them.
Also, we write $d{\bf x}_\alpha$ to indicate $dx_1\ldots dx_\alpha$. \\

1. Assume that $r,s,r',s'$ are such that
$r+s=r'+s'\leq q$, $r<r'$ and $r'+s\geq q$. Using among others the Fubini theorem, we then have
\begin{eqnarray*}
&&\langle (f \stackrel{r}{\frown} f) \stackrel{s}{\frown} f, (f\stackrel{r'}{\frown}f)\stackrel{s'}{\frown}f\rangle_{3q-2r-2s}\\
&=&\int_{\R_+^{3q}}
f({\bf a}_{q-r'},{\bf b}_{q-r'-s'},{\bf c}_{r'+s-q},{\bf x}_r)f({\bf x}_r^*,{\bf d}_{q-r-s},{\bf y}_s)
f({\bf y}_s^*,{\bf e}_{q-s})
\\
&&\hskip1cm\times
f({\bf e}^*_{q-s},{\bf d}^*_{q-r-s},{\bf c}^*_{r'+s-q},{\bf z}^*_{s'})
f({\bf z}_{s'},{\bf b}^*_{q-r'-s'},{\bf t}_{r'})
f({\bf t}_{r'}^*,{\bf a}^*_{q-r'})
\\
&&\hskip1cm\times
d{\bf x}_{r}d{\bf y}_{s}d{\bf z}_{s'}d{\bf t}_{r'}d{\bf a}_{q-r'}d{\bf b}_{q-r'-s'}d{\bf c}_{r'+s-q}d{\bf d}_{q-r-s}d{\bf e}_{q-s}\\
&=&\int_{\R_+^{3q}}
f({\bf x}_r^*,{\bf d}_{q-r-s},{\bf y}_s)
f({\bf y}_s^*,{\bf e}_{q-s})
f({\bf e}^*_{q-s},{\bf d}^*_{q-r-s},{\bf c}^*_{r'+s-q},{\bf z}^*_{s'})
\\
&&\hskip1cm\times
f({\bf z}_{s'},{\bf b}^*_{q-r'-s'},{\bf t}_{r'})
f({\bf t}_{r'}^*,{\bf a}^*_{q-r'})
f({\bf a}_{q-r'},{\bf b}_{q-r'-s'},{\bf c}_{r'+s-q},{\bf x}_r)
\\
&&\hskip1cm\times
d{\bf x}_{r}d{\bf y}_{s}d{\bf z}_{s'}d{\bf t}_{r'}d{\bf a}_{q-r'}d{\bf b}_{q-r'-s'}d{\bf c}_{r'+s-q}d{\bf d}_{q-r-s}d{\bf e}_{q-s}\\
&=&\langle (f \stackrel{s}{\frown} f) \stackrel{2q-2s-r}{\frown} f, (f\stackrel{q-r'}{\frown}f)\stackrel{q-s'}{\frown}f\rangle_{2r+2s-q}.
\end{eqnarray*}

2. Assume now that $r,s,r',s'$ are such that
$r+s=r'+s'\leq q$, $r<r'$ and $r'+s\leq q$. Similarly, we have
\begin{eqnarray*}
&&\langle (f \stackrel{r}{\frown} f) \stackrel{s}{\frown} f, (f\stackrel{r'}{\frown}f)\stackrel{s'}{\frown}f\rangle_{3q-2r-2s}\\
&=&\int_{\R_+^{3q}}
f({\bf a}_{q-r'},{\bf b}_{r'-r},{\bf x}_r)
f({\bf x}_r^*,{\bf c}_{q-s-r'},{\bf d}_{s-s'},{\bf y}_s)
f({\bf y}_s^*,{\bf e}_{q-s})
\\
&&\hskip1cm\times
f({\bf e}^*_{q-s},{\bf d}^*_{s-s'},{\bf z}^*_{s'})
f({\bf z}_{s'},{\bf c}_{q-s-r'}^*,{\bf b}^*_{r'-r},{\bf t}_{r'})
f({\bf t}_{r'}^*,{\bf a}^*_{q-r'})
\\
&&\hskip1cm\times
d{\bf x}_{r}d{\bf y}_{s}d{\bf z}_{s'}d{\bf t}_{r'}d{\bf a}_{q-r'}d{\bf b}_{r'-r}d{\bf c}_{q-s-r'}d{\bf d}_{s-s'}d{\bf e}_{q-s}\\
&=&\int_{\R_+^{3q}}
f({\bf x}_r^*,{\bf c}_{q-s-r'},{\bf d}_{s-s'},{\bf y}_s)
f({\bf y}_s^*,{\bf e}_{q-s})
f({\bf e}^*_{q-s},{\bf d}^*_{s-s'},{\bf z}^*_{s'})
\\
&&\hskip1cm\times
f({\bf z}_{s'},{\bf c}_{q-s-r'}^*,{\bf b}^*_{r'-r},{\bf t}_{r'})
f({\bf t}_{r'}^*,{\bf a}^*_{q-r'})
f({\bf a}_{q-r'},{\bf b}_{r'-r},{\bf x}_r)
\\
&&\hskip1cm\times
d{\bf x}_{r}d{\bf y}_{s}d{\bf z}_{s'}d{\bf t}_{r'}d{\bf a}_{q-r'}d{\bf b}_{r'-r}d{\bf c}_{q-s-r'}d{\bf d}_{s-s'}d{\bf e}_{q-s}\\
&=&\langle (f \stackrel{s}{\frown} f) \stackrel{q-s'}{\frown} f, (f\stackrel{q-r'}{\frown}f)\stackrel{2r'-r}{\frown}f\rangle_{q+2r-2r'}.
\end{eqnarray*}

3. Finally, assume that $r,s$ are such that
$r\leq q$, $s\leq q\wedge(2q-2r)$ and
$r+s\geq q$. We have this time
\begin{eqnarray*}
&&\|(f \stackrel{r}{\frown} f) \stackrel{s}{\frown} f\|^2_{3q-2r-2s}\\
&=&\int_{\R_+^{3q}}
f({\bf a}_{2q-2r-s},{\bf b}_{r+s-q},{\bf x}_r)
f({\bf x}_r^*,{\bf y}_{q-r})
f({\bf y}_{q-r}^*,{\bf b}^*_{r+s-q},{\bf c}_{q-s})
\\
&&\hskip1cm\times
f({\bf c}^*_{q-s},{\bf d}^*_{r+s-q},{\bf z}^*_{q-r})
f({\bf z}_{q-r},{\bf t}_{r})
f({\bf t}_{r}^*,{\bf d}_{r+s-q},{\bf a}^*_{2q-2r-s})
\\
&&\hskip1cm\times
d{\bf x}_{r}d{\bf y}_{q-r}d{\bf z}_{q-r}d{\bf t}_{r}d{\bf a}_{2q-2r-s}d{\bf b}_{r+s-q}d{\bf c}_{q-s}d{\bf d}_{r+s-q}\\
&=&\int_{\R_+^{3q}}
f({\bf x}_r^*,{\bf y}_{q-r})
f({\bf y}_{q-r}^*,{\bf b}^*_{r+s-q},{\bf c}_{q-s})
f({\bf c}^*_{q-s},{\bf d}^*_{r+s-q},{\bf z}^*_{q-r})
\\
&&\hskip1cm\times
f({\bf z}_{q-r},{\bf t}_{r})
f({\bf t}_{r}^*,{\bf d}_{r+s-q},{\bf a}^*_{2q-2r-s})
f({\bf a}_{2q-2r-s},{\bf b}_{r+s-q},{\bf x}_r)
\\
&&\hskip1cm\times
d{\bf x}_{r}d{\bf y}_{q-r}d{\bf z}_{q-r}d{\bf t}_{r}d{\bf a}_{2q-2r-s}d{\bf b}_{r+s-q}d{\bf c}_{q-s}d{\bf d}_{r+s-q}\\
&=&\langle (f \stackrel{q-r}{\frown} f) \stackrel{q-s}{\frown} f, (f\stackrel{2q-2r-s}{\frown}f)\stackrel{r}{\frown}f\rangle_{2r+2s-q}.
\end{eqnarray*}
\qed

\bigskip

We can now give the proof of Lemma \ref{lm1}. We obviously have
\begin{eqnarray*}
&&\sum_{k=\frac{q}2+1}^{\frac{3q}2-1} \left\|
\sum_{r\in\mathcal{B}_{2k}} (f_n\stackrel{r}{\frown} f_n) \stackrel{\frac{3q}2-k-r}{\frown} f_n
\right\|_{2k}^2
=\sum_{k=\frac{q}2+1}^{\frac{3q}2-1}
\sum_{r\in\mathcal{B}_{2k}} \left\|(f_n\stackrel{r}{\frown} f_n) \stackrel{\frac{3q}2-k-r}{\frown} f_n
\right\|^2_{2k} \\
&&\hskip1cm+2\sum_{k=\frac{q}2+1}^{\frac{3q}2-1}
\sum_{r<r'\in\mathcal{B}_{2k}}
\langle
(f_n \stackrel{r}{\frown} f_n) \stackrel{\frac{3q}2-k-r}{\frown} f_n, (f_n\stackrel{r'}{\frown} f_n)
\stackrel{\frac{3q}2-k-r'}{\frown} f_n\rangle_{2k},
\end{eqnarray*}
so we are left to show that
\begin{equation}\label{toprove}
2\sum_{k=\frac{q}2+1}^{\frac{3q}2-1}
\sum_{r<r'\in\mathcal{B}_{2k}}
\langle
(f_n \stackrel{r}{\frown} f_n) \stackrel{\frac{3q}2-k-r}{\frown} f_n, (f_n\stackrel{r'}{\frown} f_n)
\stackrel{\frac{3q}2-k-r'}{\frown} f_n\rangle_{2k}=
2\sum_{k=0}^{\frac{q}2-1} \left\|
\sum_{r\in\mathcal{B}_{2k}} (f_n\stackrel{r}{\frown} f_n) \stackrel{\frac{3q}2-k-r}{\frown} f_n
\right\|^2_{2k}.
\end{equation}
To achieve this goal, let us decompose the left-hand side of (\ref{toprove}) as follows:
\begin{eqnarray*}
2\sum_{k=\frac{q}2+1}^{\frac{3q}2-1}
\sum_{r<r'\in\mathcal{B}_{2k}} (\ldots)
=2\sum_{l=1}^{\frac{q}2} \sum_{k=\frac{q}2+l}^{\frac{3q}2-l}
\sum_{\stackrel{r\in\mathcal{B}_{2k} \mbox{ \tiny{s.t.}}}{\stackrel{r+l\in\mathcal{B}_{2k}}{}}}
{\bf 1}_{\{r'=r+l\}} (\ldots)+
2\sum_{l=1}^{\frac{q}2} \sum_{r'>r+l \in\mathcal{B}_{q+2l}} (\ldots)
=(1)+(2).
\end{eqnarray*}
(The second sum in (1) finishes at $\frac{3q}2-l$ instead of $\frac{3q}2-1$, because $r'=r+l$ with $r,r'\in\mathcal{B}_{2k}$ implies $k\leq \frac{3q}2-l$.)
Using Lemma \ref{lm1} (point 2), we have
\begin{eqnarray*}
(1)&=&2\sum_{l=1}^{\frac{q}2}
\sum_{k=\frac{q}2+l}^{\frac{3q}2-l}
\sum_{\stackrel{r\in\mathcal{B}_{2k}\mbox{ \tiny{s.t.}}}{\stackrel{r+l\in\mathcal{B}_{2k}}{}}} \langle (f_n \stackrel{r}{\frown} f_n) \stackrel{\frac{3q}{2}-k-r}{\frown}f_n, (f_n \stackrel{r+l}{\frown} f_n) \stackrel{\frac{3q}{2}-k-r-l}{\frown}f_n \rangle_{2k} \\
&=&2\sum_{l=1}^{\frac{q}2}
\sum_{k=\frac{q}2+l}^{\frac{3q}2-l}
\sum_{\stackrel{r\in\mathcal{B}_{2k}\mbox{ \tiny{s.t.}}}{\stackrel{r+l\in\mathcal{B}_{2k}}{}}}
\langle
(f_n \stackrel{\frac{3q}2-k-r}{\frown} f_n) \stackrel{k+r+l-\frac{q}{2}}{\frown} f_n,
(f_n\stackrel{q-r-l}{\frown} f_n) \stackrel{r+2l}{\frown} f_n
\rangle_{q-2l}\\
&=&2\sum_{l=1}^{\frac{q}{2}} \sum_{s\leq s'\in\mathcal{B}_{q-2l}}
\langle
(f_n \stackrel{s}{\frown} f_n) \stackrel{q+l-s}{\frown} f_n,
(f_n\stackrel{s'}{\frown} f_n) \stackrel{q+l-s'}{\frown} f_n
\rangle_{q-2l}\\
&=&\sum_{l=1}^{\frac{q}2}\left(
\sum_{s\in\mathcal{B}_{q-2l}}
\left\|
(f_n \stackrel{s}{\frown} f_n) \stackrel{q+l-s}{\frown} f_n
\right\|_{q-2l}^2
+\left\|
\sum_{s\in\mathcal{B}_{q-2l}} (f_n \stackrel{s}{\frown} f_n) \stackrel{q+l-s}{\frown} f_n
\right\|_{q-2l}^2
\right)\\
&=&\sum_{k=0}^{\frac{q}2-1}\left(
\sum_{r\in\mathcal{B}_{2k}}
\left\|
(f_n \stackrel{r}{\frown} f_n) \stackrel{\frac{3q}2-k-r}{\frown} f_n
\right\|_{2k}^2
+\left\|
\sum_{r\in\mathcal{B}_{2k}} (f_n \stackrel{r}{\frown} f_n) \stackrel{\frac{3q}2-k-r}{\frown} f_n
\right\|_{2k}^2
\right).
\end{eqnarray*}
On the other hand, Lemma \ref{lm1} (point 1) leads to
\begin{eqnarray*}
(2)&=&2\sum_{l=1}^{\frac{q}2}
\sum_{\stackrel{r,r'\in\mathcal{B}_{q+2l}}{\stackrel{r'>r+l}{}}}
\langle
(f_n \stackrel{r}{\frown} f_n) \stackrel{q-l-r}{\frown} f_n,
(f_n\stackrel{r'}{\frown} f_n) \stackrel{q-l-r'}{\frown} f_n
\rangle_{q+2l}\\
&=& 2\sum_{l=1}^{\frac{q}2}
\sum_{\stackrel{r,r'\in\mathcal{B}_{q+2l}}{\stackrel{r'>r+l}{}}}
\langle
(f_n \stackrel{q-l-r}{\frown} f_n) \stackrel{2l+r}{\frown} f_n,
(f_n\stackrel{q-r'}{\frown} f_n) \stackrel{l+r'}{\frown} f_n
\rangle_{q-2l}\\
&=&2\sum_{l=1}^{\frac{q}{2}} \sum_{s< s'\in\mathcal{B}_{q-2l}}
\langle
(f_n \stackrel{s}{\frown} f_n) \stackrel{q+l-s}{\frown} f_n,
(f_n\stackrel{s'}{\frown} f_n) \stackrel{q+l-s'}{\frown} f_n
\rangle_{q-2l}\\
&=&2\sum_{k=0}^{\frac{q}2-1}
\sum_{r<r'\in\mathcal{B}_{2k}}
\langle
(f_n \stackrel{r}{\frown} f_n) \stackrel{\frac{3q}2-k-r}{\frown} f_n,
(f_n \stackrel{r'}{\frown} f_n) \stackrel{\frac{3q}2-k-r'}{\frown} f_n
\rangle_{q-2l}.
\end{eqnarray*}

Finally,
\[
(1)+(2) = 2\sum_{k=0}^{\frac{q}2-1} \left\|
\sum_{r\in\mathcal{B}_{2k}} (f_n \stackrel{r}{\frown} f_n) \stackrel{\frac{3q}2-k-r}{\frown} f_n
\right\|^2_{2k},
\]
and the proof of Lemma \ref{lm1} is concluded. \qed

\medskip

\noindent
{\bf Acknowledgement.} We thank Samy Tindel for suggesting us the name of the hero of this paper...
the tetilla law!
We thank an anonymous referee for his/her constructive remarks.
We are grateful to Roland Speicher for helpful comments and for bringing the reference \cite{duke} to our notice,
to G\'erard Le Caer for having guessed the formula (\ref{mom-exp}) and to Jean-S\'ebastien Giet
for helping us to prove it rigourously.


\begin{thebibliography}{99}

\bibitem{amv} H. Airault, P. Malliavin and F.G. Viens (2010).
Stokes formula on the Wiener space and $n$-dimensional Nourdin-Peccati analysis. {\it J. Funct. Anal.}
{\bf 258}, 1763-1783.




\bibitem{bbnp} H. Bierm\'e, A. Bonami, I. Nourdin and G. Peccati (2011). Optimal Berry-Esseen rates on the Wiener space: the barrier of
third and fourth cumulants. Preprint.

\bibitem{bianespeicher} P. Biane and R. Speicher (1998). Stochastic analysis with respect to free Brownian motion and analysis on Wigner space. {\it Probab. Theory Rel. Fields} {\bf 112}, 373-409.


\bibitem{donati}
\rm C. Donati-Martin (2003).
\rm Stochastic integration with respect to $q$-Brownian motion.
{\it Probab. Theory Rel. Fields} {\bf 125}, 77--95.

\bibitem{edelman}
\rm P. H. Edelman (1980).
\rm Chain enumeration and noncrossing partitions.
{\it Discrete Math.} {\bf 31}, 2, 171--180.


\bibitem{knps} T. Kemp, I. Nourdin, G. Peccati and R. Speicher (2011). Wigner chaos and the fourth moment. {\it Ann. Probab.}, in press.

\bibitem{duke}
\rm A. Nica and R. Speicher (1998).
\rm Commutators of free random variables.
{\it Duke Math. J.} {\bf 92}, 553--592.


\bibitem{nicaspeicher}
\rm A. Nica and R. Speicher (2006). {\it Lectures on the Combinatorics of Free Probability}. Cambridge University Press.

\bibitem{yetanother} I. Nourdin (2010). Yet another proof of the Nualart-Peccati criterion.
{\it Electron. Comm. Probab.}, to appear.

\bibitem{WWW} I. Nourdin (2010). {\it A webpage on Stein's method, Malliavin calculus
and related topics}. Http address: {\tt www.iecn.u-nancy.fr/$\sim$nourdin/steinmalliavin.html}

\bibitem{np-ptrf} I. Nourdin and G. Peccati (2009). Stein's method on Wiener chaos. {\it Probab. Theory Rel. Fields} {\bf 145}(1), 75-118.

\bibitem{np-aop} I. Nourdin and G. Peccati (2009). Stein's method and exact Berry-Esseen asymptotics for functionals of Gaussian fields.
{\it Ann. Probab.} {\bf 37}, no. 6, 2231-2261.

\bibitem{np-poisson} I. Nourdin and G. Peccati (2011). Poisson approximations on the free Wigner chaos. Preprint.

\bibitem{NPbook} I. Nourdin and G. Peccati (2012). {\it Normal approximations with Malliavin calculus. From Stein's method to universality}. Cambridge University Press, to appear.



\bibitem{npr-jfa}
\rm I. Nourdin, G. Peccati and G. Reinert (2009).
\rm Second order Poincar\'{e} inequalities and CLTs on Wiener space.
{\it J. Func. Anal.} {\bf 257}, 593-609.



\bibitem{npr-aop}
\rm I. Nourdin, G. Peccati and G. Reinert (2010).
\rm Invariance principles for homogeneous sums: universality of Gaussian Wiener chaos.
\rm {\it Ann. Probab.} {\bf  38}, no. 5, 1947-1985.


\bibitem{nrp}
\rm I. Nourdin, G. Peccati and A. R\'eveillac (2010).
\rm Multivariate normal approximation using Stein's method and Malliavin calculus.
{\it Ann. Inst. H. Poincar\'e Probab. Statist.} {\bf 46}, no. 1, 45-58




\bibitem{nv-ejp}
\rm I. Nourdin and F.G. Viens (2009).
\rm Density estimates and concentration inequalities with Malliavin calculus.
{\it Electron. J. Probab.} {\bf 14}, 2287-2309.




\bibitem{Nualart-book}
\rm D. Nualart (2006).
\it The Malliavin calculus and related topics.
\rm Springer Verlag, Berlin, Second edition.

\bibitem{NP} D. Nualart and G. Peccati (2005). Central limit theorems for sequences of multiple stochastic integrals. \it Ann. Probab. {\bf 33} \rm (1), 177-193.

\bibitem{voiculescu} D.V. Voiculescu (1991). Limit laws for random matrices and free product. {\it Invent. Math.} {\bf 104}, 201--220.

\end{thebibliography}
\end{document}